\documentclass[11pt]{amsart}
\usepackage{latexsym, mathrsfs, color, tikz,tikz-cd,multirow,bbm,mathtools, amsmath, amssymb, comment,xcolor}
\usetikzlibrary{decorations.markings}

\usepackage{graphics}
\usepackage{subcaption}
\usepackage{xparse}
\input{xy}
\xyoption{all}
\usepackage{adjustbox}
\usepackage[colorlinks=true, pdfstartview=FitV,
 linkcolor=blue,anchorcolor=Periwinkle, citecolor=red,urlcolor=Emerald]{hyperref}

\usetikzlibrary{matrix, positioning}
\usetikzlibrary{arrows, arrows.meta}
\tikzset{->-/.style={decoration={  markings,  mark=at position #1 with
			{\arrow{>}}},postaction={decorate}}}
\tikzset{-<-/.style={decoration={  markings,  mark=at position #1 with
			{\arrow{<}}},postaction={decorate}}}
\tikzcdset{scale cd/.style={every label/.append style={scale=#1},
    cells={nodes={scale=#1}}}}

\newenvironment{red}{\relax\color{red}}{\relax}
\newenvironment{blue}{\relax\color{blue}}{\hspace*{.5ex}\relax}

\newcommand{\ber}{\begin{red}}
\newcommand{\er}{\end{red}}
\newcommand{\beb}{\begin{blue}}
\newcommand{\eb}{\end{blue}}

\theoremstyle{plain}
\newtheorem{theorem}{Theorem}[section]
\newtheorem{thmx}{Theorem}

\newtheorem{lemma}[theorem]{Lemma}
\newtheorem{corollary}[theorem]{Corollary}

\theoremstyle{definition}
\newtheorem{definition}[theorem]{Definition}

\newtheorem{example}[theorem]{Example}
\newtheorem{remark}[theorem]{Remark}
\newtheorem{notations}[theorem]{Notations}
\numberwithin{equation}{section}

\def\surf{\mathbf{S}}                       
\def\TT{\mathbf{T}}
\def\PP{\mathbf{P}}
\def\MM{\mathbf{M}}

\def\<{\langle}
\def\>{\rangle}

\renewcommand{\k}{\mathbf{k}}

 \renewcommand{\mod}{\operatorname{mod}}

\newcommand{\Int}{\operatorname{Int}}
\newcommand{\Hom}{\operatorname{Hom}}
\newcommand{\Ext}{\operatorname{Ext}}
\newcommand{\udim}{\operatorname{\underline{dim}}}

\begin{document}
\title[A geometric model for the non-$\tau$-rigid modules of type $\widetilde{D}_n$]{A geometric model for the non-$\tau$-rigid modules of type $\widetilde{D}_n$}
\author[Blake Jackson]{Blake Jackson}
\address{Department of Mathematics, University of Connecticut,
	Storrs, CT 06269, U.S.A.}
\email{blake.jackson@uconn.edu}


\begin{abstract}
    We give a geometric model for the non-$\tau$-rigid modules over acyclic path algebras of type $\widetilde{D}_n$. Similar models have been provided for module categories over path algebras of types $A_n, D_n,$ and $\widetilde{A}_n$ as well as the $\tau$-rigid modules of type $\widetilde{D}_n$. A major draw of these geometric models is the ``intersection-dimension formulas'' they often come with. These formulas give an equality between the intersection number of the curves representing the modules in the geometric model and the dimension of the extension spaces between the two modules. This formula allows us to calculate the homological data between two modules combinatorially. Since there are infinitely many distinct homogeneous stable tubes in the regular component of the Auslander-Reiten quiver of type $\widetilde{D}_n$, all of which are disjoint, our geometric data requires an extra decoration on the admissible edges in our geometric model to prevent intersections between curves corresponding to modules in distinct stable tubes of the Auslander-Reiten quiver.
\end{abstract}

\maketitle

\section{Introduction}

Cluster categories were introduced in \cite{buan_cluster-tilted_2007,buan_cluster_2008,buan_tilting_2006} to study the relationship between quiver representations and the cluster algebras introduced by Fomin and Zelevinsky \cite{fomin_cluster_2002,fomin_cluster_2003}.
A cluster category is a modified version of the module category of a hereditary algebra (in fact, the cluster category is the orbit category of the bounded derived category of finitely-generated modules over a path algebra under the action of a specific automorphism).
In \cite{caldero_quivers_2006}, the authors realize the cluster categories of type $A_n$ by constructing a category of diagonals of a regular polygon with $n+3$ vertices.
Based on the polygon model for cluster algebras of type $A_n$, Fomin, Shapiro, and Thurston classify all cluster algebras arising from triangulations of orientable bordered marked surfaces \cite{fomin_cluster_2008}.
Later, Schiffler \cite{schiffler_geometric_2008} provided a model for cluster (module) categories of type $D_n$ by using triangulations of a once-punctured $(n+2)$-gon.
More recently, Baur and Torkildsen \cite{baur_geometric_2015} give a geometric model for the modules over path algebras of type $\widetilde{A}_n$ and He, Zhou, and Zhu \cite{he_mutation_2022,he_geometric_2023} provided a geometric model for modules over skew-gentle algebras. 
Since cluster algebras of type $E$ cannot be realized via triangulations of a surface, the combinatorial models for type $E$ are not as clean as the other Dynkin types; however, a model has been provided by Lamberti \cite{lamberti_combinatorial_2015}.

Another active area of research is the geometric modeling of $m$-cluster categories
\cite{jacquet-malo_geometric_2025,torkildsen_geometric_2015,zhou_cluster_2009}.
These higher analogs of the cluster category recover the classical cluster categories when $m=1$.
Again, $m$-cluster categories have been studied for types $A_n, D_n, \widetilde{A}_n$, and $\widetilde{D}_n$.

The shortcoming of these previous works is that most of them (except \cite{baur_geometric_2015}) neglect a large family of modules over path algebras of Euclidean type (those with a $\widetilde{tilde}$): the regular modules.
Tilting theory studies \textit{rigid} and \textit{$\tau$-rigid} indecomposable modules; however, not all indecomposable modules are rigid (or even $\tau$-rigid).
Our contribution is a geometric model for the \textit{non-$\tau$-rigid} modules over acyclic path algebras of type $\widetilde{D}_n$, which completes the geometric modeling for modules over path algebras of Euclidean quivers from surfaces.

The main feature of these geometric models regarding the module category is the so-called intersection-dimension formula for the modules over these algebras.
Modules in the geometric model are represented by some set of admissible decorated curves in their respective surface. 
The intersection-dimension formula relates the geometric data (intersection number of curves) with the homological data of the modules they represent (dimension of the first extension spaces between the two modules).
The former is a straightforward calculation after the model is given; the latter can be highly nontrivial, especially for path algebras over quivers of affine Dynkin type (for example, quivers of type $\widetilde{D}_n$).
This paper continues this tradition with an intersection-dimension formula for the non-$\tau$-rigid modules of type $\widetilde{D}_n$.

The main result of the paper is the following theorem:
\begin{thmx} \label{main-thm}
    Let $\surf$ be a triangulated, twice-punctured disk with $(n-2)$ marked points on the boundary whose triangulation $\TT$ corresponds to an acyclic quiver $Q^\TT$ of type $\widetilde{D}_n$, $\k$ be an algebraically closed field, and $\k Q^\TT$ be the path algebra over $Q^\TT$. 
    Then given any two colored admissible tagged edges $(\gamma_1, \kappa_1, \lambda_1)$ and $(\gamma_2, \kappa_2, \lambda_2)$ (not necessarily distinct), 
    $$\Int((\gamma_1, \kappa_1, \lambda_1), (\gamma_2, \kappa_2, \lambda_2)) = \dim_\k \Ext^1 (M_1, M_2) + \dim_\k \Ext^1 (M_2, M_1)$$ where $M_i = M(\gamma_i,\kappa_i, \lambda_i)$ is a $\k Q^\TT$-module for $i =1,2$, $\Int$ is the intersection number between two admissible tagged edges, and $\Ext^1(M,N)$ is the group of extensions of $N$ by $M$ viewed as a $\k$-vector space spanned by short exact sequences in $\mod \k Q^\TT$.
\end{thmx}

The paper is organized as follows.
Section 2 gives background information on quivers and path algebras. 
Section 3 describes the module category over a path algebra of affine Dynkin type.
Sections 4 and 5 are dedicated to the category of colored admissible tagged edges in a twice-punctured disk.
Section 6 gives an equivalence of categories, and Section 7 completes the proof of Theorem~\ref{main-thm} by giving the intersection-dimension formula.

Throughout this paper, $\k$ is an algebraically closed field.
For any $\k$-algebra $\mathcal{A}$, we consider only finite-dimensional left $\mathcal{A}$-modules where $\mod \mathcal{A}$ is the abelian category of these modules.
Finally, we use the convention that the affine Dynkin diagram of type $\widetilde{D}_n$ contains $n+1$ vertices.

\section{Quivers and Path Algebras}

\begin{definition}\label{def-quiver}
    A \textbf{quiver} $Q = (Q_0, Q_1)$ is a finite digraph without loops and directed 2-cycles where $Q_0$ is the set of vertices and $Q_1$ is the set of arrows.
    The elements of $Q_0$ are indexed by the numbers $1, 2, ..., n$. 
    If $\alpha \in Q_1$ with $\alpha: i \to j$, then we say $s(\alpha) = i$ is the \textbf{source} of $\alpha$ and $t(\alpha) = j$ is the \textbf{target} of $\alpha$.
    A quiver $Q$ is \textbf{acyclic} if it contains no directed cycles of any length.
\end{definition}

\begin{definition}
    Let $i \in Q_0$ be a vertex in $Q$. 
    Then $\mu_i Q$ is the \textbf{mutation} of $Q$ at vertex $i$ and is the quiver obtained from $Q$ in the following way:
    \begin{enumerate}
        \item for each path $j \to i \to k$ in $Q$ of length 2 passing through $i$, add an arrow $j \to k$ in $\mu_i Q$
        \item reverse all arrows which begin or end at vertex $i$
        \item delete any 2-cycles that have appeared as a result of step 1.
    \end{enumerate}
    A quiver $Q'$ is \textbf{mutation-equivalent} to $Q$ if there is a finite sequence of vertices $i_1, i_2, ..., i_k$ in $Q_0$ such that $Q' = \mu_{i_1}\mu_{i_2}...\mu_{i_k}Q$.
    If $G$ is a finite or affine Dynkin diagram, then a quiver $Q$ is \textbf{of type $G$} if $Q$ is mutation equivalent to an acyclic orientation of $G$. 
\end{definition}

Quivers were introduced by Gabriel \cite{gabriel_unzerlegbare_1972} and have become fundamental objects in the study of cluster algebras developed by Fomin and Zelevinsky \cite{fomin_cluster_2002,fomin_cluster_2003}.
While this paper is less concerned with cluster theory, developments in the field of cluster algebras have led to significant contributions in other areas of mathematics, such as the representation theory of algebras.
This paper is concerned primarily with quivers of type $\widetilde{D}_n$; however, care is taken to state results in the most general sense.
It is also worth noting that mutation acts as an involution on quivers (and clusters \cite{fomin_cluster_2002}).
This means that for a quiver $Q$, $\mu_i \mu_i Q = Q$.

\begin{definition}
    Let $Q$ be a finite, connected quiver and $\k$ be an algebraically closed field.
    Then $\k Q$ is the \textbf{path algebra over $Q$} with basis given by the set of all directed paths in $Q$ and multiplication between two basis elements given by path concatenation.
    If $Q$ is acyclic then the basis of $\k Q$ will have finitely many elements; in this case, the algebra is \textbf{finite-dimensional}.
    If $G$ is a finite or affine Dynkin diagram, then $\k Q$ is \textbf{of type $G$} if $Q$ is of type $G$. 
\end{definition}

\begin{example}
Perhaps the first example we should consider is the path algebra of type $A_2$.
In this case, the quiver is given by
\[\begin{tikzpicture}
        \node[scale=1] at (-1, 0) {$Q = A_2 = \quad$};
        \filldraw[black] (0,0) circle (2pt);
        \node[scale=1] at (0, 0.4) {1};
        \filldraw[black] (2,0) circle (2pt);
        \node[scale=1] at (2, 0.4) {2};
        
        \draw[thick,->, shorten >=5, shorten <=5, >=stealth'](0,0)to(2,0);
        \node[scale=1] at (1, 0.3) {$\alpha$};
\end{tikzpicture}\]
Then $\k Q = \k A_2$ has vector space basis $\{ e_1, \alpha, e_2 \}$, where $e_i$ is the lazy or constant path at vertex $i$.
If we write our paths right-to-left, the multiplication between basis elements is \[\alpha e_1 = \alpha, \quad e_2\alpha = \alpha, \quad e_1^2 = e_1, \quad e_2^2 = e_2,\] and all other products are zero.
Therefore, our path algebra is isomorphic to the following upper triangular matrix algebra:
\[\k Q = \k A_2 \cong \left\{ \begin{bmatrix}
    \lambda_{e_1} & \lambda_{\alpha} \\
    0 & \lambda_{e_2}
\end{bmatrix} \ : \ \lambda_c \in \k \text{ for } c \text{ a path} \right\} \]
\end{example}

We care about path algebras over quivers because of the following classical result from the representation theory of finite-dimensional algebras.

\begin{theorem}
    Every finite-dimensional algebra over an algebraically closed field is Morita equivalent to a quotient of a path algebra over some quiver $Q$.
\end{theorem}

Moreover, if the algebra is hereditary, we can drop the ``a quotient of'' condition.
A ring or an algebra is said to be \textbf{hereditary} if all submodules of projective modules are projective.
Hereditary algebras are algebras with global dimension 1.
This already tells us a great deal about the homological properties of these algebras: if $\mathcal{A}$ is a hereditary algebra, then $\Ext^i_{\mathcal{A}}(M,N) = 0$ for $i > 1$ and any $M,N \in \mod \mathcal{A}$.
This restricts the study of extensions to $\Ext^1_{\mathcal{A}}$ in hereditary algebras.
Since the path algebras we are interested in (acyclic type $\widetilde{D}_n$) are hereditary, this motivates the right-hand side of Theorem~\ref{main-thm}.

\section{The Category of $\widetilde{D}_n$-modules}
This section defines the first category featured in Theorem~\ref{main-thm}--the category of modules of type $\widetilde{D}_n$. 
Preprojective module categories of algebras have been studied extensively \cite{assem_elements_2006,auslander_representation_1995,draxler_existence_1996,gabriel_representations_1997,geis_rigid_2006,geis_rigid_2007,ringel_tame_1984,simson_elements_2007}.
Due to the well-documented nature of the preprojective categories, this section mainly serves as a brief overview of topics relevant to the present paper and a deeper description of the regular modules of type $\widetilde{D}_n$.
Before we continue, it will be helpful to have an equivalent characterization of the modules over the path algebra of an acyclic quiver.

\begin{definition}
    Let $Q$ be a quiver and $\k$ be an algebraically closed field.
    A \textbf{representation} $M = (M_i, \varphi_\alpha)_{i \in Q_0, \alpha \in Q_1}$ of $Q$ is a collection of $\k$-vector spaces $M_i$ (one for each vertex in $Q_0$) and a collection of $\k$-linear maps $\varphi_\alpha: M_{s(\alpha)} \to M_{t(\alpha)}$ (one for each arrow in $Q_1$).
    The representation is \textbf{finite-dimensional} if each $M_i$ is. 
    If $M$ is finite-dimensional, the \textbf{dimension vector} $\udim M$ of $M$ is the vector $(\dim M_i)_{i \in Q_0}$ of the dimensions of the vector spaces at each vertex.

    If $M = (M_i, \varphi_\alpha)$ and $M' = (M'_i, \varphi'_\alpha)$ are two representations of $Q$, then the \textbf{direct sum} of $M$ and $M'$ is \[ M \oplus M' = \left( M_i \oplus M_i', \begin{bmatrix}
        \varphi_\alpha & 0 \\
        0 & \varphi_\alpha'
    \end{bmatrix} \right)_{i \in Q_0, \alpha \in Q_1} \] A representation $M$ of Q is \textbf{indecomposable} if it is nonzero and cannot be written as the direct sum of two nonzero representations.

    If $M = (M_i, \varphi_\alpha)$ and $M' = (M'_i, \varphi'_\alpha)$ are two representations of $Q$, a \textbf{homomorphism} of representations $f: M \to M'$ is a collection $(f_i)_{i\in Q_0}$ of linear maps $f_i: M_i \to M'_i$ such that for each arrow $\alpha: i \to j$ in $Q_1$, we have commutative diagrams $f_j \circ \varphi_\alpha(m) = \varphi'_\alpha \circ f_i(m)$ for all $m\in M_i$.
    The abelian category of all finite-dimensional representations of $Q$ with morphisms given by representation homomorphisms is denoted by $\operatorname{rep} Q$.
\end{definition}

\begin{theorem}[\cite{schiffler_quiver_2014}, Theorem 5.4]\label{reps-modules}
    Let $Q$ be a finite, connected, acyclic quiver. 
    Then, the finite-dimensional representations of $Q$ are in bijection with the finitely generated $\k Q$ modules (up to isomorphism). 
    This bijection also applies to their homomorphisms and respects the composition of these homomorphisms.
\end{theorem}

Most modern work dealing with the representation theory of hereditary algebras $\mathcal{A} = \k Q$ make no distinction between a $\k Q$-module and its corresponding quiver representation.

Our goal by the end of this section is to understand the global and local structure of the category $\mod \k Q$.
The objects of this category are isoclasses of finitely-generated, indecomposable (left) $\k Q$ modules, and the morphisms of the category are compositions of module homomorphisms between $\k Q$ modules.
The most basic of the morphisms are the \textit{irreducible morphisms}, which we now define.

\begin{definition}\label{def-AR}
    Let $\mathcal{A} = \k Q$ be an irreducible hereditary algebra (i.e., $Q$ only has a single connected component) and let $A, B, C, M$ and $N$ be $\mathcal{A}$-modules.
    A morphism $h$ (dually $g$) in $\mod\mathcal{A}$ is a \textbf{section} (dually \textbf{retraction}) if $h$ is a right (left) inverse of some morphism in $\mod\mathcal{A}$.
    A morphism $f: A \to B$ in $\mod\mathcal{A}$ is called \textbf{irreducible} if
    \begin{enumerate}
        \item $f$ is not a section nor a retraction
        \item whenever $f = gh$ for some morphisms $h: A \to C$ and $g: C \to B$, then either $h$ is a section or $g$ is a retraction.
    \end{enumerate}
\end{definition}

We now define the almost split sequences in $\mod \ kQ$.

\begin{definition}
    A short exact sequence in $\mod\mathcal{A}$ $$0 \to A \to B \to C \to 0$$ is \textbf{split} if $B \cong A \oplus C$.
    A short exact sequence in $\mod\mathcal{A}$ $$0 \to N \xrightarrow{h} B \xrightarrow{g} M \to 0$$ is \textbf{almost split} (or is an \textbf{Auslander-Reiten sequence}) if $M$ and $N$ are indecomposable and $h$ and $g$ are irreducible morphisms.
    In this case, $N$ is uniquely determined and $N \cong \tau M$ where $\tau$ is the \textbf{Auslander-Reiten translation}.
\end{definition}

The projective objects in $\mod \k Q$ are precisely the projective modules. 
These modules (and their dual injective modules) play an important role in the global structure of $\mod \k Q$, so we define them below.

\begin{definition}[\cite{schiffler_quiver_2014}]
    Let $Q$ be an acyclic quiver. 
    The indecomposable \textbf{projective module} at vertex $i$, denoted $P(i)$, has the quiver representation \[ P(i) = (P(i)_j, \varphi_\alpha)_{j \in Q_0, \alpha \in Q_1} \]
    where $P(i)_j$ is the $\k$-vector space with basis the set of all paths from $i$ to $j$ in $Q$.
    The indecomposable \textbf{injective module} at vertex $i$, denoted $I(i)$, has the quiver representation \[ I(i) = (I(i)_j, \varphi_\alpha)_{j \in Q_0, \alpha \in Q_1} \]
    where $I(i)_j$ is the $\k$-vector space with basis the set of all paths from $j$ to $i$ in $Q$.
\end{definition}

If $M$ is not projective ($N$ is not injective), then there exists an Auslander-Reiten sequence ending at $M$ and starting at $\tau M$ (starting at $N$ and ending at $\tau^{-1}N$).
We are now ready to define the 3 major components of the Auslander-Reiten quiver for the modules over an affine-type Dynkin quiver.
A ``big picture'' diagram of the Auslander-Reiten quiver of type $\widetilde{D}_n$ can be found in Figure~\ref{fig:modules category}.

\begin{definition}
    Let $P(j)$ be the indecomposable projective module at vertex $j$, $I(j)$ be the indecomposable injective module at vertex $j$, and $\tau$ be the Auslander-Reiten translation. 
    Then a $\k Q$-module $M$ is called \textbf{preprojective} if $\tau^{i}M \cong P(j)$ for some $j \in Q_0$ and $i \geq 0$, \textbf{preinjective} if $\tau^{-i}M \cong I(j)$ for some $j \in Q_0$ and $i \geq 0$, and \textbf{regular} if $\tau^{i}M \cong M$ for some $i \geq 1$.
\end{definition}

\begin{figure}[h]
\centering
\resizebox{\textwidth}{!}{
\begin{tikzpicture}
\draw[thick] (4,5) -- (0,5) -- (2,3) -- (1,2) -- (3,0) -- (4,0);
\node[rotate=-45,scale=1.2] at (0.8,3.6) {Projective};
\node[rotate=-45,scale=1.2] at (1.5,1) {Modules};

\node at (4,2.5) {$\bullet$};
\node at (4.2,2.5) {$\bullet$};
\node at (4.4,2.5) {$\bullet$};

\draw[thick] (6,5) ellipse [x radius=0.5, y radius=0.2];
\draw[thick] (5.5,5) -- (5.5,0);
\draw[thick] (6.5,5) -- (6.5,0);
\draw[thick] (7.5,5) ellipse [x radius=0.5, y radius=0.2];
\draw[thick] (7,5) -- (7,0);
\draw[thick] (8,5) -- (8,0);
\draw[thick] (9.5,5) ellipse [x radius=1, y radius=0.2];
\draw[thick] (8.5,5) -- (8.5,0);
\draw[thick] (10.5,5) -- (10.5,0);

\node at (11.3,2.5) {$\bullet$};
\node at (11.5,2.5) {$\bullet$};
\node at (11.7,2.5) {$\bullet$};

\draw[thick] (13,5) ellipse [x radius=0.3, y radius=0.2];
\draw[thick] (12.7,5) -- (12.7,0);
\draw[thick] (13.3,5) -- (13.3,0);
\draw[thick] (14,5) ellipse [x radius=0.3, y radius=0.2];
\draw[thick] (13.7,5) -- (13.7,0);
\draw[thick] (14.3,5) -- (14.3,0);

\node at (15.2,2.5) {$\bullet$};
\node at (15,2.5) {$\bullet$};
\node at (14.8,2.5) {$\bullet$};

\draw[thick] (15,0) -- (20,0) -- (18,2) -- (19,3) -- (17,5) -- (15,5);
\node[rotate=-45,scale=1.2] at (18.5,4) {Injective};
\node[rotate=-45,scale=1.2] at (19,1.5) {Modules};

\end{tikzpicture}}
\caption{A diagram of the Auslander-Reiten quiver for modules of type $\widetilde{D}_n$}
\label{fig:modules category}
\end{figure}
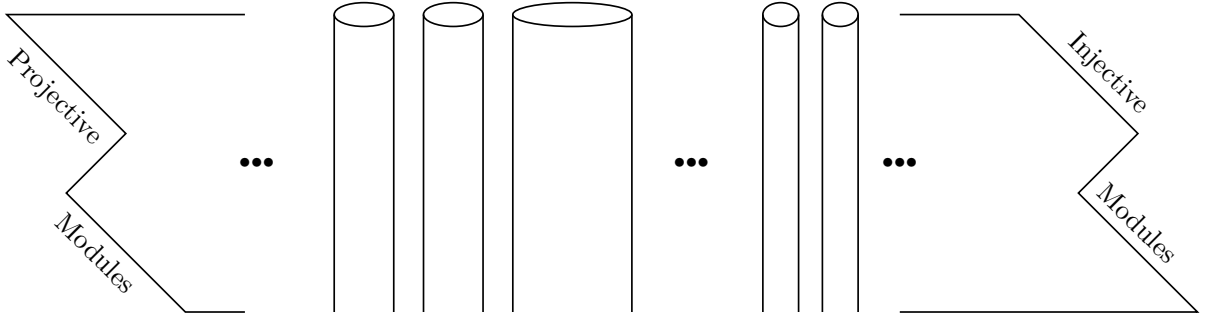

\begin{definition}
    Let $\mathcal{A} = \k Q$ be an irreducible hereditary algebra. Then the \textbf{Auslander-Reiten quiver} $\Gamma(\mod \mathcal{A} )$ is defined as follows:
    \begin{itemize}
        \item The vertices of $\Gamma( \mod \mathcal{A} )$ are the isomorphism classes $[M]$ of indecomposable modules $M$ in $\mod \mathcal{A}$.
        \item There is an arrow $[M] \to [N]$ in $\Gamma( \mod \mathcal{A} )$ if and only if there is an irreducible morphism $M \to N$ in $\mod \mathcal{A}$.
    \end{itemize}
    Auslander-Reiten sequences form the \textbf{meshes} of $\Gamma(\mod \mathcal{A})$.
    If $\mathcal{A}$ is a representation-infinite hereditary algebra, then the \textbf{preprojective component} of $\Gamma( \mod \mathcal{A} )$ is denoted by $\mathcal{P(A)}$ and is the unique connected component of $\Gamma( \mod \mathcal{A} )$ that contains all indecomposable projective $\mathcal{A}$-modules, the \textbf{preinjective component} of $\Gamma( \mod \mathcal{A} )$ is denoted by $\mathcal{Q(A)}$ and is the unique connected component of $\Gamma( \mod \mathcal{A} )$ that contains all indecomposable injective $\mathcal{A}$-modules, and the \textbf{regular component} of $\Gamma( \mod \mathcal{A} )$ is denoted by $\mathcal{R(A)}$ and is the disjoint union of stable tubes $\{ \mathcal{T}^{\lambda} \}_{\lambda \in \mathbb{P}_1}$ which contain all indecomposable regular $\mathcal{A}$-modules.
\end{definition}

The Auslander-Reiten translation serves as a horizontal right-to-left translation through the Auslander-Reiten quiver.

For path algebras of acyclic type $A_n, D_n,$ and $E_i$ for $i = 6,7,8$, the Auslander-Reiten quiver is finite and connected. 
This is because these algebras are of finite representation type, and this classification is due to Gabriel's Theorem \cite{gabriel_unzerlegbare_1972}.
In this case, all modules are preprojective (and preinjective), and the regular component is empty.

If the path algebra is of acyclic type $\widetilde{A}_n, \widetilde{D}_n,$ or $\widetilde{E}_i$ for $i = 6,7,8$, then the situation is more complicated.
If the quiver is of type $\widetilde{D}_n$, the Auslander-Reiten quiver will have 3 distinct components: the preprojective component containing all of the projective modules, the preinjective component containing all of the injective modules, and the regular component, which is a disjoint union of infinitely many ``stable tubes''.
Each of these components has a similar local structure: the meshes.
However, the global structure of the regular component is unlike the preprojective and preinjective components, which are dual to one another. 
The geometric model presented in this paper is for the entire module category. 
There are three types of meshes that appear in an Auslander-Reiten quiver of type $\widetilde{D}_n$ ($n > 4$), these are displayed in Figure~\ref{fig:meshes}. 
There is one more mesh that is found in Auslander-Reiten quivers of type $\widetilde{D}_4$ that looks like the mesh in the bottom right of Figure~\ref{fig:meshes}, but has 4 middle terms instead of 3.

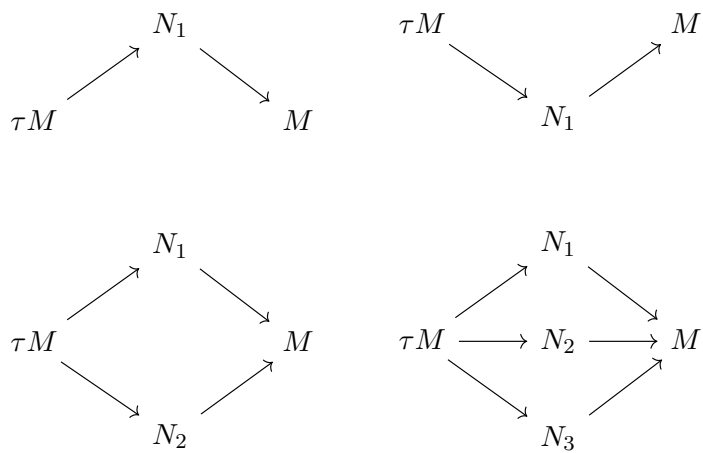
\begin{figure}[h]
    \[\begin{tikzcd}
     & N_1 \arrow{dr} & \\
     \tau M\arrow{ur} &  & M
    \end{tikzcd}\hspace{8mm}
    \begin{tikzcd}
     \tau M\arrow{dr} &  & M \\
     & N_1 \arrow{ur} & 
    \end{tikzcd}\]
    \vspace{8mm}
    \[\begin{tikzcd}
     & N_1 \arrow{dr} & \\
     \tau M\arrow{ur}\arrow{dr} &  & M \\
      & N_2\arrow{ur} &
    \end{tikzcd}\hspace{8mm}
    \begin{tikzcd}
     & N_1 \arrow{dr} & \\
     \tau M\arrow{ur}\arrow{dr}\arrow{r} & N_2\arrow{r} & M \\
      & N_3\arrow{ur} &
    \end{tikzcd}\]  
    \caption{Meshes of the Auslander-Reiten quiver of type $\widetilde{D}_n$ ($n > 4$).}
    \label{fig:meshes}
\end{figure}

We now introduce a fundamental result, the Auslander-Reiten formulas, and brief definitions of the operations that appear in them.
For a more in-depth treatment of these topics, see \cite{assem_elements_2006,auslander_representation_1995,schiffler_quiver_2014}. In what follows, for representations $M,N$, let $\Hom(M,N)$ be the $\k$-vector space of all homomorphism from $M$ to $N$ and $\Ext^1(M,N)$ be the $\k$-vector space of extensions of $M$ by $N$. 

\begin{definition}
    Let $Q^{op}$ be the quiver obtained from $Q$ by reversing the direction of all the arrows.
    The \textbf{duality} $$D = \Hom_\k (-,\k): \operatorname{rep} Q \to \operatorname{rep} Q^{op}$$ is the contravariant functor defined as follows:
    \begin{itemize}
        \item For representations $M = (M_i, \varphi_\alpha)$, we have $$DM = (DM_i, D\varphi_{\alpha^{op}})_{i\in Q_0, \alpha \in Q_1},$$ where $DM_i$ is the dual vector space and if $\alpha$ is an arrow in $Q$ then $D\varphi_{\alpha^{op}}$ is the pullback of $\varphi_\alpha$ \begin{align*}
            D\varphi_{\alpha^{op}}: DM_{t(\alpha)} &\to DM_{s(\alpha)} \\
                                    u &\mapsto u \circ \varphi_\alpha.
        \end{align*}
        \item For homomorphisms $f:M \to N$ in $\operatorname{rep} Q$, we have $Df : DN \to DM$ in $\operatorname{rep} Q^{op}$ defined by $Df(u) = u \circ f$.
    \end{itemize}
\end{definition}

\begin{definition} \label{linehome}
    Let $P(M,N)$ be the set of all homomorphisms $f \in \Hom(M,N)$ such that $f$ factors through a projective $\mathcal{A}$-module, and define $$\underline\Hom(M,N) = \Hom(M,N)/P(M,N).$$
    Dually, let $I(M,N)$ be the set of all homomorphisms $f \in \Hom(M,N)$ such that $f$ factors through an injective $\mathcal{A}$-module, and define $$\overline\Hom(M,N) = \Hom(M,N)/I(M,N).$$
\end{definition}

We now state the Auslander-Reiten formulas.

\begin{theorem}[Auslander-Reiten formulas]\label{them-AR}
    Let $M,N$ be $\mathcal{A}$-modules. 
    Then there are isomorphisms $$\Ext^1(M,N) \cong D \underline{\Hom} (\tau^{-1}N, M) \cong D \overline{\Hom} (N, \tau M).$$
    If $\mathcal{A}$ is a hereditary algebra (i.e., $\mathcal{A} \cong \k Q$ for some acyclic quiver $Q$), then $\underline\Hom(M,N) = \overline\Hom(M,N) = \Hom(M,N)$ and the formulas become $$\Ext^1(M,N) \cong D \Hom (\tau^{-1}N, M) \cong D \Hom (N, \tau M).$$
\end{theorem}

The most immediate corollary of Theorem~\ref{them-AR} is as follows:

\begin{corollary} \label{ext-hom}
    Let $M,N$ be $\mathcal{A}$-modules. Then $$\dim_\k \Ext^1(M,N) = \dim_\k \Hom(\tau^{-1}N, M) = \dim_\k \Hom (N,\tau M).$$
\end{corollary}

The following corollary and theorem will be very helpful in simplifying the proof of the main theorem.

\begin{corollary} \label{ext-tau-invariant}
    Let $M,N$ be non-projective $\k Q$-modules for some acyclic quiver $Q$. Then $$\Ext^1(\tau M,\tau N) \cong \Ext^1(M,N).$$
\end{corollary}

\begin{proof}
    By the Auslander-Reiten formulas, $$\Ext^1(\tau M,\tau N) \cong D \Hom (\tau^{-1} \tau N, \tau M) =  D \Hom (N, \tau M) \cong \Ext^1(M,N).$$
\end{proof}

\begin{theorem}\label{thm-0-homs}
    Let $\Gamma(\mod \k Q )$ be the Auslander-Reiten quiver of the module category of the path algebra over an acyclic quiver of type $\widetilde{D}_n$. 
    Then \[ \Hom (\mathcal{Q}(\k Q) ,\mathcal{P}(\k Q)) = \Hom (\mathcal{R}(\k Q) ,\mathcal{P}(\k Q)) = \Hom (\mathcal{Q}(\k Q) ,\mathcal{R}(\k Q)) = 0. \]
\end{theorem}

Thanks to Corollary~\ref{ext-tau-invariant} and Theorem~\ref{thm-0-homs}, we can simplify the exposition quite a bit.
Previous work has handled the case where $M$ and $N$ are preprojective/preinjective. 
This leaves the cases where at least one of the two modules is regular. However, the preprojective and preinjective components are dual to one another, Corollary~\ref{ext-tau-invariant} says we can apply $\tau$ to any preprojective module until it becomes projective, and there are no nonzero homomorphisms going from the regular component to the preprojective component.
Therefore, we need only consider the case when $M$ is regular and $N$ is projective or regular.

We close this section with some motivating examples. 

\begin{example}\label{running-exa}
    Let $Q$ be the following canonically oriented quiver of type $\widetilde{D}_6$
    
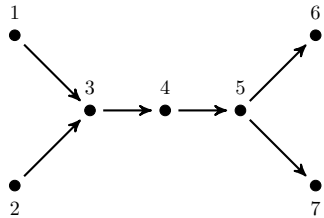
\begin{center}
    \begin{tikzpicture}
        \filldraw[black] (0,0) circle (2pt);
        \node[scale=0.7] at (0, 2.3) {1};
        \filldraw[black] (0,2) circle (2pt);
        \node[scale=0.7] at (0, -0.3) {2};
        \filldraw[black] (1,1) circle (2pt);
        \node[scale=0.7] at (1, 1.3) {3};
        \filldraw[black] (2,1) circle (2pt);
        \node[scale=0.7] at (2, 1.3) {4};
        \filldraw[black] (3,1) circle (2pt);
        \node[scale=0.7] at (3, 1.3) {5};
        \filldraw[black] (4,2) circle (2pt);
        \node[scale=0.7] at (4, 2.3) {6};
        \filldraw[black] (4,0) circle (2pt);
        \node[scale=0.7] at (4, -0.3) {7};
        
        \draw[thick,->, shorten >=5, shorten <=5, >=stealth'](0,2)to(1,1);
        \draw[thick,->, shorten >=5, shorten <=5, >=stealth'](0,0)to(1,1);
        \draw[thick,->, shorten >=5, shorten <=5, >=stealth'](1,1)to(2,1);
        \draw[thick,->, shorten >=5, shorten <=5, >=stealth'](2,1)to(3,1);
        \draw[thick,->, shorten >=5, shorten <=5, >=stealth'](3,1)to(4,2);
        \draw[thick,->, shorten >=5, shorten <=5, >=stealth'](3,1)to(4,0);
    \end{tikzpicture}
    \end{center}
    The path algebra of this quiver is also isomorphic to an upper triangular matrix algebra. 
    Below is the projective representation at vertex 3 $P(3)$, the injective representation at vertex 3 $I(3)$, and another family of modules we will call $M_\lambda$.
    The maps without labels are understood to be either $0$ or the identity so that the representation is indecomposable. 
    The modules $M_\lambda$ are indecomposable for all choices of $\lambda \in \mathbb{P}^1(\k) = \mathbb{P}^1$ (the choice of $\lambda = \infty$ corresponds to the map $[0 \ 1]:\k^2 \to \k$).
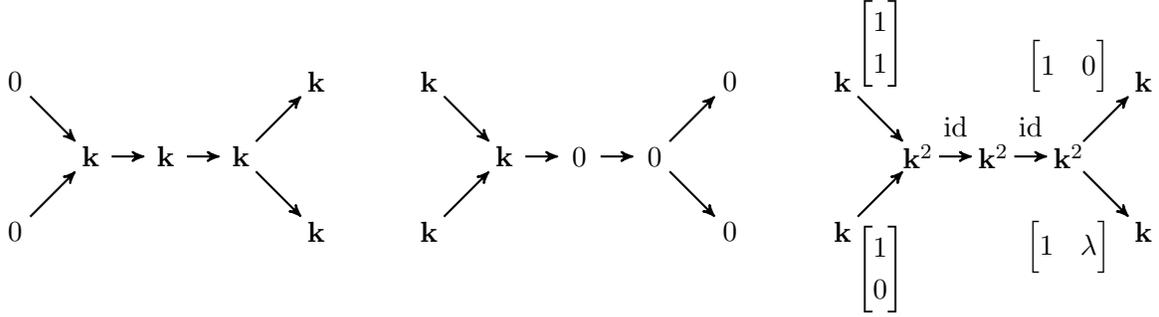
    \begin{center}
    \begin{tikzpicture}
    \begin{scope}[shift={(0,0)}]
        \node at (0,2) {$0$};
        \node at (0,0) {$0$};
        \node at (1,1) {$\k$};
        \node at (2,1) {$\k$};
        \node at (3,1) {$\k$};
        \node at (4,2) {$\k$};
        \node at (4,0) {$\k$};
        
        \draw[thick,->, shorten >=8, shorten <=8, >=stealth'](0,2)to(1,1);
        \draw[thick,->, shorten >=8, shorten <=8, >=stealth'](0,0)to(1,1);
        \draw[thick,->, shorten >=8, shorten <=8, >=stealth'](1,1)to(2,1);
        \draw[thick,->, shorten >=8, shorten <=8, >=stealth'](2,1)to(3,1);
        \draw[thick,->, shorten >=8, shorten <=8, >=stealth'](3,1)to(4,2);
        \draw[thick,->, shorten >=8, shorten <=8, >=stealth'](3,1)to(4,0); 
    \end{scope}
    \begin{scope}[shift={(5.5,0)}]
        \node at (0,2) {$\k$};
        \node at (0,0) {$\k$};
        \node at (1,1) {$\k$};
        \node at (2,1) {$0$};
        \node at (3,1) {$0$};
        \node at (4,2) {$0$};
        \node at (4,0) {$0$};
        
        \draw[thick,->, shorten >=8, shorten <=8, >=stealth'](0,2)to(1,1);
        \draw[thick,->, shorten >=8, shorten <=8, >=stealth'](0,0)to(1,1);
        \draw[thick,->, shorten >=8, shorten <=8, >=stealth'](1,1)to(2,1);
        \draw[thick,->, shorten >=8, shorten <=8, >=stealth'](2,1)to(3,1);
        \draw[thick,->, shorten >=8, shorten <=8, >=stealth'](3,1)to(4,2);
        \draw[thick,->, shorten >=8, shorten <=8, >=stealth'](3,1)to(4,0); 
    \end{scope}
    \begin{scope}[shift={(11,0)}]
        \node at (0,2) {$\k$};
        \node at (0,0) {$\k$};
        \node at (1,1) {$\k^2$};
        \node at (2,1) {$\k^2$};
        \node at (3,1) {$\k^2$};
        \node at (4,2) {$\k$};
        \node at (4,0) {$\k$};

        \node at (0.5,2.5) {$\begin{bmatrix}
            1 \\
            1
        \end{bmatrix}$};
        \node at (0.5,-0.5) {$\begin{bmatrix}
            1 \\
            0
        \end{bmatrix}$};
        \node at (1.5,1.4) {id};
        \node at (2.5,1.4) {id};
        \node at (3,2.2) {$\begin{bmatrix}
            1 & 0
        \end{bmatrix}$};
        \node at (3,-0.2) {$\begin{bmatrix}
            1 & \lambda
        \end{bmatrix}$};
        
        \draw[thick,->, shorten >=8, shorten <=8, >=stealth'](0,2)to(1,1);
        \draw[thick,->, shorten >=8, shorten <=8, >=stealth'](0,0)to(1,1);
        \draw[thick,->, shorten >=8, shorten <=8, >=stealth'](1,1)to(2,1);
        \draw[thick,->, shorten >=8, shorten <=8, >=stealth'](2,1)to(3,1);
        \draw[thick,->, shorten >=8, shorten <=8, >=stealth'](3,1)to(4,2);
        \draw[thick,->, shorten >=8, shorten <=8, >=stealth'](3,1)to(4,0); 
    \end{scope}
    \end{tikzpicture}
    \end{center}
    The reader should verify that $\Hom(P(3), M_\lambda) \cong \Hom(M_\lambda, I(3)) \cong \k^2$ and $\Hom(M_{\lambda_1}, M_{\lambda_2}) = 0$ if $\lambda_1 \neq \lambda_2$ where $\lambda_i$ is considered to be a point in $\mathbb{P}^1$.
    There are also 3 special values of $\lambda$: $0,1,\infty$.
    Note that the 3 fixed non-identify linear maps in $M_\lambda$ determine 3 unique lines in $\k^2$: the images of the two maps on the left are the lines $x=0$ and $y=x$, and the kernel of the map in the upper right is the line $y=0$.
    These correspond to the special values $0,1,$ and $\infty$ (respectively) of $\lambda$.
    
    We also leave it as an exercise to show that $\tau^2 M_0 = M_0$, $\tau^2 M_1 = M_1$, $\tau^4 M_\infty = M_\infty$, and $\tau M_\lambda = M_\lambda$ for all other values of $\lambda \in \mathbb{P}^1$ and that $\tau^{-i} P(3) \neq 0$ and $\tau^i I(3) \neq 0$ for all $i \geq 0$.
    These examples can be extended to a general quiver of type $\widetilde{D}_n$.
    The only change that needs to be made in the general case is $\tau^{n-2} M_\infty = M_\infty$; the rest of the statements are the same.
    This is because, for a quiver of type $\widetilde{D}_n$, there are 3 exceptional stable tubes of rank 2, 2, and $(n-2)$ called $\mathcal{T}^0, \mathcal{T}^1,$ and $\mathcal{T}^\infty$ which contain the irreducible modules $M_0, M_1,$ and $M_\infty$ respectively, and an infinite number of homogeneous stable tubes (those of rank 1) $\mathcal{T}^\lambda$ which are uniquely determined by the indecomposable modules $M_\lambda$ for $\lambda \neq 0,1,\infty$.
    See \cite{simson_elements_2007-1} for more details.
\end{example}

Our goal in this paper is to present a geometric model that captures the $\tau$-, $\Hom$-, and $\Ext$-combinatorics of the modules over a path algebra of type $\widetilde{D}_n$ (i.e., those in the example above).
We would also like our geometric model to agree with the geometric models for the $\tau$-rigid modules; the $\tau$-rigid modules of type $\widetilde{D}_n$ are the preprojective, preinjective, and some of the regular modules.
Some of the regular modules that sit near the mouth of the stable tubes are $\tau$-rigid.
Therefore, in the next few sections, our goal is to give a punctured disk model for the modules over a path algebra of type $\widetilde{D}_n$.

\section{Main Construction and Tagged Rotation}

We now lay out the geometric construction and the second important category of Theorem~\ref{main-thm}: the category of tagged edges. 
The underlying geometric model for quivers and their mutations arising from triangulated surfaces is due to Fomin, Shapiro, and Thurston \cite{fomin_cluster_2008}.
Others have described a correspondence between module categories (or cluster categories) of types $A_n, D_n, \widetilde{A}_n$, and the $\tau$-rigid modules of type $\widetilde{D}_n$ and the category of distinguished arcs in some triangulated surfaces \cite{baur_geometric_2015,caldero_quivers_2006,he_mutation_2022, he_geometric_2023, schiffler_geometric_2008}.
However, a geometric model of the non-$\tau$-rigid modules of type $\widetilde{D}_n$ has remained elusive. 
These modules are often overlooked because they are not $\tau$-tilting objects, so tilting theory does not apply to them. 
We address this gap in the literature by extending the geometric model and its intersection-dimension formula to the non-$\tau$-rigid (regular) modules of type $\widetilde{D}_n$.

\begin{definition}\label{def-tagged}
    A \textbf{punctured, marked surface with boundary} as defined in \cite{fomin_cluster_2008} is a triple $\surf = (S, \MM, \PP)$ where $S$ is an oriented surface with boundary, $\PP \in S \setminus \partial S$ is the set of punctures, and $\MM \in \partial S$ is the set of marked points on the boundary of $S$.
    An \textbf{admissible tagged edge} or \textbf{curve} $(\gamma, \kappa)$ in a punctured, marked surface with boundary $\surf$ is a continuous map $\gamma: [0,1] \to \surf$ and a map $\kappa: \{ t \mid \gamma(t) \in \PP \} \to \{0,1\}$ such that 
    \begin{enumerate}
        \item $\gamma(0), \gamma(1) \in \PP \cup \MM$,
        \item $\gamma(t) \in \Delta^0 = \surf \setminus (\PP \cup \partial S) \text{ for } 0 < t < 1$,
        \item if $\gamma(0) \in \MM$ or $\gamma(1) \in \MM$, $\gamma$ or its self-intersection do not cut out a once-punctured monogon (Figure~\ref{fig:monogon}),
        \item $\gamma$ is not homotopic to a boundary segment of $\surf$, and
        \item $\gamma$ is not null-homotopic.
    \end{enumerate}
    Let $E^\times$ be the set of admissible tagged edges $(\gamma, \kappa)$. 
    If $\gamma$ has both endpoints in $\MM$, then the domain of $\kappa$ is $\emptyset$ and for convenience we write $\kappa = \emptyset$.
\end{definition}

\begin{figure}[h]
    \centering
    \begin{tikzpicture}
    \draw[red,thick,postaction={decorate,decoration={
        markings, mark=at position 0.33 with {\arrow{>}}}}] plot [smooth, tension=1] coordinates {(3,0) (3.5,1) (3,3) (2.5,1) (3,0)};
    
    \draw[red,thick,postaction={decorate,decoration={
        markings, mark=at position 0.33 with {\arrow{>}}}}] plot [smooth, tension=1] coordinates {(6,0) (6.5,1) (6,3) (5.5,1) (7,1.7)};
        
    \draw[thick] (0,0) -- (9,0);
    
    \filldraw[black] (3,0) circle (2pt);
    \filldraw[black] (6,0) circle (2pt);
    
    \filldraw[black] (3,2) circle (2pt);
    \filldraw[black] (6,2) circle (2pt);
    
    \end{tikzpicture}
    \caption{Non-admissible curves cutting out once-punctured monogons.}
    \label{fig:monogon}
\end{figure}
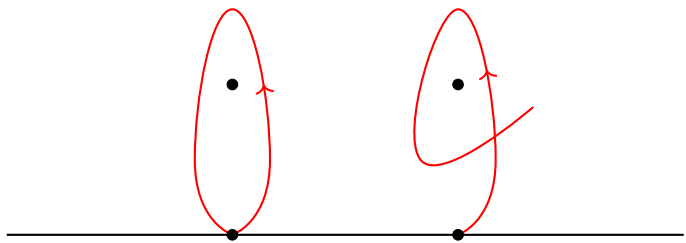

\begin{remark}
    For us, $\surf$ is always a disk, but this doesn't have to be true in general.
    For any admissible tagged edge with one endpoint in $\PP$ and the other in $\MM$, we will assume that $\gamma(0) \in \PP$.
    For a twice-punctured surface, label the two punctures $P_1$ and $P_2$.
    We will assume that all curves with both endpoints in distinct punctures have $\gamma(0) = P_1$ and $\gamma(1) = P_2$; this simplifies the exposition.
    Also, note that condition (3) in Definition~\ref{def-tagged} allows for curves to cut out once punctured monogons if both endpoints of $\gamma$ are in $\PP$.
\end{remark}

\begin{notations}
    Graphically, the map $\kappa$ is indicated by using a notch or bowtie. 
    If $\gamma(0) \in \PP$ and $\kappa(0) = 1$, then $\gamma$ will have a notch or bowtie drawn near $\gamma(0)$. 
    Otherwise, the edge will be plain near $\gamma(0)$.
    For example, see Figure~\ref{fig:punct-intersection}.
\end{notations}

The existence of edges ending in a puncture complicates the notion of intersection between two admissible tagged edges.

\begin{definition}\label{def-puctintersection}
    Two admissible tagged edges $(\gamma_1,\kappa_1), (\gamma_2,\kappa_2) \in E^\times$ are said to \textbf{intersect in a puncture} if $t_1, t_2 \in \{0,1\}$ and 
    \begin{enumerate}
        \item $\gamma_1(t_1) = \gamma_2(t_2) \in \PP$
        \item $\kappa_1(t_1) \neq \kappa_2(t_2)$
        \item If $\gamma_1$ and $\gamma_2$ are homotopic as untagged edges, then $\gamma_1(1 - t_1) = \gamma_2(1 - t_2) \in \PP$ and $\kappa_1(1-t_1) \neq \kappa_2(1-t_2)$.
    \end{enumerate}
\end{definition}

In other words, there are three cases that result in punctured intersections.
Two admissible tagged edges starting at the same puncture and ending at two different points in $\MM$ have 1 punctured intersection if they have different tags at the puncture. 
If two admissible tagged edges are homotopic as untagged edges and have different tagging at each end, they will have 2 punctured intersections, one for each endpoint in a puncture.
Finally, if two admissible tagged edges have both endpoints in punctures but are not homotopic as untagged edges, they will have a punctured intersection at each puncture where they have opposite tagging.
The left column of Figure~\ref{fig:punct-intersection} shows pairs of curves with no punctured intersections, and the right column shows a pair of curves with 1 punctured intersection (top right panel) and 2 punctured intersections (bottom right panel).

\begin{figure}[h]
    \centering
    \begin{tikzpicture}
        \begin{scope}[shift={(0,2)}] 
            \draw[red,thick, postaction={decorate,decoration={markings,mark=at position 0.8 with {\node[rotate=90, text=red, scale=1] at (0,0) {$\bowtie$};}}}] plot [smooth, tension=1] coordinates {(0,0) (2,0)};
            \draw[red,thick, postaction={decorate,decoration={markings,mark=at position 0.8 with {\node[rotate=90, text=red, scale=1] at (0,0) {$\bowtie$};}}}] plot [smooth, tension=1] coordinates {(4,0) (2,0)};
            \filldraw[black] (2,0) circle (3pt);
        \end{scope}

        \begin{scope}[shift={(5,2)}] 
            \draw[red,thick] plot [smooth, tension=1] coordinates {(0,0) (2,0)};
            \draw[red,thick, postaction={decorate,decoration={markings,mark=at position 0.8 with {\node[rotate=90, text=red, scale=1] at (0,0) {$\bowtie$};}}}] plot [smooth, tension=1] coordinates {(4,0) (2,0)};
            \filldraw[black] (2,0) circle (3pt);
        \end{scope}

        \begin{scope}[shift={(0,0)}] 
            \draw[red,thick, postaction={decorate,decoration={markings,mark=at position 0.9 with {\node[rotate=45, text=red, scale=1] {$\bowtie$};}}}] plot [smooth, tension=1] coordinates {(0,0) (2,1) (4,0)};
            \draw[red,thick] plot [smooth, tension=1] coordinates {(0,0) (2,-1) (4,0)};
            \filldraw[black] (0,0) circle (3pt);
            \filldraw[black] (4,0) circle (3pt);
        \end{scope}

        \begin{scope}[shift={(5,0)}] 
            \draw[red,thick, postaction={decorate,decoration={markings,mark=at position 0.9 with {\node[rotate=45, text=red, scale=1] {$\bowtie$};}}}] plot [smooth, tension=1] coordinates {(0,0) (2,1) (4,0)};
            \draw[red,thick, postaction={decorate,decoration={markings,mark=at position 0.9 with {\node[rotate=45, text=red, scale=1] {$\bowtie$};}}}] plot [smooth, tension=1] coordinates {(4,0) (2,-1) (0,0)};
            \filldraw[black] (0,0) circle (3pt);
            \filldraw[black] (4,0) circle (3pt);
        \end{scope}
    \end{tikzpicture}
    \caption{Left column: no punctured intersections; Right column: punctured intersections}
    \label{fig:punct-intersection}
\end{figure}
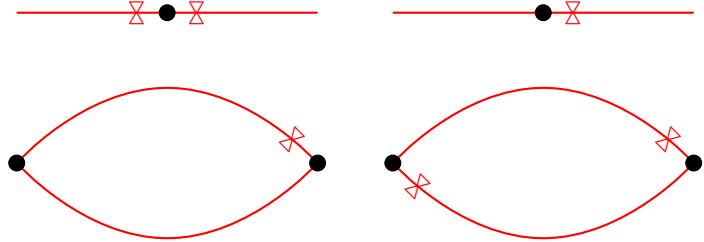

\subsection{Colored Admissible Tagged Edges}
There is a lot of subtlety that arises when studying the regular modules of a Euclidean path algebra. 
As we saw in Example~\ref{running-exa}, the regular modules of type $\widetilde{D}_n$ are partitioned into families of stable tubes $\mathcal{T}^\lambda$.
It is known that there are no nonzero homomorphisms between regular modules in distinct tubes.
However, the geometric model we have laid out so far will not be able to pick up on this subtle detail. 
To account for this, we add another decoration to our admissible tagged edges so that edges in distinct stable tubes will not intersect.
This additional decoration is a coloring on the edges.

\begin{definition}\label{def-coloring}
    Let $(\gamma,\kappa) \in E^\times$ be an admissible tagged edge. 
    A \textbf{coloring} of $(\gamma,\kappa)$ is a triple $(\gamma,\kappa,\lambda)$ where $\lambda \in \mathbb{P}^1(\k) \cup \{ -\infty \}$.
    If no coloring is specified, assume that the edge is colored with the distinguished color $-\infty$\footnote{For the sake of simplicity, assume that $-\infty < \lambda$ for all $\lambda\in\mathbb{P}^1$. Having a total order on $\mathbb{P}^1 \cup \{ -\infty \}$ will improve the exposition of the paper and allow us to color our edges using the parameter $\lambda$ from Example~\ref{running-exa}}.
    That is, $(\gamma,\kappa) = (\gamma,\kappa,-\infty)$ unless otherwise indicated.
    In this way, we can view all admissible tagged edges as colored, even if it is not explicitly stated.
    Let $E^\times_\lambda$ denote the set of colored admissible tagged edges $(\gamma,\kappa,\lambda)$.
\end{definition}

\begin{remark}
    Since the phrase ``colored admissible tagged edge'' is a bit of a mouthful, we will prefer to say ``admissible edge'' when there is no risk of confusion.
\end{remark}

We will end up coloring all the edges of a triangulation (defined later), the preprojective edges, and the preinjective edges (defined later) with the color $-\infty$, and each of the stable tubes of edges (defined later) will be assigned a distinct color $\lambda \in \mathbb{P}^1$.
Now, we can define the intersection of two admissible edges.

\begin{definition}
    Let $(\gamma_1,\kappa_1,\lambda_1), (\gamma_2,\kappa_2,\lambda_2) \in E^\times_\lambda$ be two admissible edges and $\Delta^0 = \surf \setminus (\PP \cup \partial S)$. Then, if $\lambda_1 = \lambda_2$ or $\text{min}(\lambda_1,\lambda_2) = -\infty$, the \textbf{intersection number} between these admissible edges is 
    $$\Int( (\gamma_1,\kappa_1,\lambda_1), (\gamma_2,\kappa_2,\lambda_2) ) \coloneqq
        \text{min}\{ \text{Card} (\gamma_1 \cap \gamma_2 \cap \Delta^0 ) \} + \text{Card}( \mathfrak{P}( (\gamma_1,\kappa_1), (\gamma_2,\kappa_2) ) )$$ 
    where $ \mathfrak{P}( (\gamma_1,\kappa_1), (\gamma_2,\kappa_2) )$ counts the number of punctured intersections, Card is the cardinality of a set, and all curves $\gamma_1, \gamma_2$ are considered up to homotopy relative to their endpoints.
    Otherwise, if $-\infty \neq \lambda_1 \neq \lambda_2 \neq -\infty$, we set $\Int( (\gamma_1,\kappa_1,\lambda_1), (\gamma_2,\kappa_2,\lambda_2) ) = 0$.
    Two admissible edges \textbf{cross} or \textbf{intersect} if $\Int( (\gamma_1,\kappa_1,\lambda_1), (\gamma_2,\kappa_2,\lambda_2) ) > 0$.
    Intersections that occur in $\Delta^0$ are called \textbf{normal intersections}.
    When counting the self-intersections of a curve, assume that we have two homotopic copies of the same curve that are only allowed to intersect transversely.
\end{definition}

Note that intersections are possible only if at least one of the edges is colored $-\infty$ or if both edges share the same color.

\begin{definition}\label{def-tri}
    Let $\surf$ be a punctured, marked surface with boundary. 
    A (maximal/ideal) \textbf{triangulation} $\TT$ of $\surf$ is a maximal collection of non-crossing admissible edges in $\surf$ (with color $-\infty$). 
    Note that ``non-crossing'' excludes admissible edges with self-crossings and punctured intersections from the triangulation.
\end{definition}

\begin{definition}\label{def-qt}
    Let $\surf$ be a punctured, marked surface with boundary and $\TT$ be a triangulation of $\surf$. 
    Then $Q^\TT$, the \textbf{quiver associated to} $\TT$, is the following quiver:
    \begin{itemize}
        \item The elements in $Q_0^\TT$ are in bijection with the admissible edges $(\gamma,\kappa, -\infty) \in \TT$
        \item There is an edge $i \to j$ in $Q_1^\TT$ if and only if 
        \begin{enumerate}
            \item the admissible edges representing $i$ and $j$ in $\TT$ share a common endpoint in $\alpha_0 \in \MM \cup \PP$
            \item $j$ is the direct counter-clockwise neighbor of $i$ at $\alpha_0$
        \end{enumerate}
    \end{itemize}
    The final step in forming $Q^\TT$ is to pairwise delete any 2-cycles; these will be generated by edges ending at a puncture.
    If $G$ is a finite or affine Dynkin diagram, then $\TT$ is \textbf{of type $G$} if $Q^\TT$ is of type $G$. 
    A triangulation $\TT$ is said to be \textbf{acyclic} if $Q^\TT$ is. 
    We only consider acyclic triangulations of type $\widetilde{D}_n$.
\end{definition}

\begin{remark}
    The quiver associated with a triangulation $\TT$ is finite if $\TT$ is finite.
    Moreover, mutation of the quiver at vertex $i$ can be realized geometrically by the so-called ``flip'' of the admissible edge $i$ (see \cite{fomin_cluster_2008}). 
    Figure~\ref{fig:example-triangulation} gives an example of non-acyclic triangulation of type $\widetilde{D}_n$.
    Note that the edges labeled 1 and 2 share an endpoint in a puncture, so there are arrows from $2 \to 1$ and $1 \to 2$ in $Q^\TT$; however, these are pairwise deleted.
\end{remark}

\begin{figure}[h]
    \centering
    \begin{tikzpicture}
    \begin{scope}[shift={(0,0)}]

        \draw[red,thick, postaction={decorate,decoration={markings,mark=at position 0.9 with {\node[rotate=30, text=red, scale=0.6] {$\bowtie$};}}}] plot [smooth, tension=1] coordinates {(2,3) (1,2) (1.3,1)};
        \node[text=red, scale=0.7] at (0.7, 1.5) {1};
        \draw[red,thick] plot [smooth, tension=1] coordinates {(2,3) (1.3,1)};
        \node[text=red, scale=0.7] at (1.3, 1.6) {2};
        \draw[red,thick] plot [smooth, tension=1] coordinates {(2,3) (1.5,0.5) (0,1)};
        \node[text=red, scale=0.7] at (1.3, 0.2) {3};
        \draw[red,thick] plot [smooth, tension=1] coordinates {(2,3) (2,-1)};
        \node[text=red, scale=0.7] at (1.8, 0) {4};
        \draw[red,thick] plot [smooth, tension=1] coordinates {(2,3) (2.5,0.5) (4,1)};
        \node[text=red, scale=0.7] at (2.7, 0.2) {5};
        \draw[red,thick] plot [smooth, tension=1] coordinates {(2,3) (2.7,1)};
        \node[text=red, scale=0.7] at (2.7,1.6) {6};
        \draw[red,thick, postaction={decorate,decoration={markings,mark=at position 0.9 with {\node[rotate=-30, text=red, scale=0.6] {$\bowtie$};}}}] plot [smooth, tension=1] coordinates {(2,3) (3,2) (2.7,1)};
        \node[text=red, scale=0.7] at (3.3,1.5) {7};
        
        \filldraw[black] (2,3) circle (2pt);
        \filldraw[black] (4,1) circle (2pt);
        \filldraw[black] (2,-1) circle (2pt);
        \filldraw[black] (0,1) circle (2pt);

        \filldraw[black] (1.3,1) circle (2pt);
        \filldraw[black] (2.7,1) circle (2pt);

        \draw[thick] (2,1) circle (2);
    \end{scope}
    
      \begin{scope}[shift={(6,0)}]
        \filldraw[black] (0,0) circle (2pt);
        \node[scale=0.7] at (0, 2.3) {1};
        \filldraw[black] (0,2) circle (2pt);
        \node[scale=0.7] at (0, -0.3) {2};
        \filldraw[black] (1,1) circle (2pt);
        \node[scale=0.7] at (1, 1.3) {3};
        \filldraw[black] (2,1) circle (2pt);
        \node[scale=0.7] at (2, 1.3) {4};
        \filldraw[black] (3,1) circle (2pt);
        \node[scale=0.7] at (3, 1.3) {5};
        \filldraw[black] (4,2) circle (2pt);
        \node[scale=0.7] at (4, 2.3) {6};
        \filldraw[black] (4,0) circle (2pt);
        \node[scale=0.7] at (4, -0.3) {7};
        
        \draw[thick,->, shorten >=5, shorten <=5, >=stealth'](0,2)to(1,1);
        \draw[thick,->, shorten >=5, shorten <=5, >=stealth'](0,0)to(1,1);
        \draw[thick,->, shorten >=5, shorten <=5, >=stealth'](1,1)to(2,1);
        \draw[thick,->, shorten >=5, shorten <=5, >=stealth'](2,1)to(3,1);
        \draw[thick,->, shorten >=5, shorten <=5, >=stealth'](3,1)to(4,2);
        \draw[thick,->, shorten >=5, shorten <=5, >=stealth'](3,1)to(4,0);
      \end{scope}
    \end{tikzpicture}
    \caption{A triangulation of type $\widetilde{D}_6$ with along with $Q^\TT$.}
    \label{fig:example-triangulation}
\end{figure}
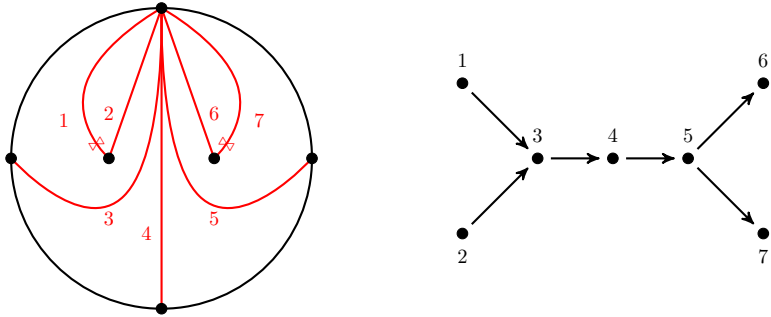

We now define the one-end shift operation [1] and the poliwhirl operation $\vartheta$.
These are the combinatorial moves that are in bijection with the ``reverse'' irreducible morphisms in the Auslander-Reiten quiver.

\begin{definition}\label{def-rotation}
    For a curve $\gamma$ in $\surf$ with $\gamma(0) \in \MM$ (resp. $\gamma(1) \in \MM$), let $\gamma[1]$ (resp. $[1]\gamma$) be the curve obtained from $\gamma$ by moving $\gamma(0)$ (resp. $\gamma(1)$) along the boundary counterclockwise to the next marked point. 
    If $\gamma(0) \in \PP$ and $\gamma(1) \in \MM$, then $\gamma[1] = \gamma$ and $[1]\gamma$ is the curve obtained from $\gamma$ by moving $\gamma(1)$ along the boundary counterclockwise to the next marked point.
\end{definition}

Figure~\ref{fig:endshift} gives an example of the $[1]$ operator acting on an admissible edge with at least one endpoint in the boundary.

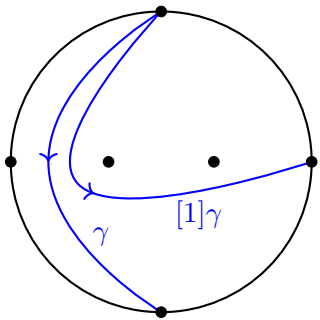
\begin{figure}[h]
    \begin{tikzpicture}
        \draw[blue,thick, postaction={decorate,decoration={markings,mark=at position 0.5 with {\arrow{>}}}}] plot [smooth, tension=1] coordinates {(2,3) (0.5,1) (2,-1)};
        \node[blue,scale=1] at (1.2, 0) {$\gamma$};
        \draw[blue,thick, postaction={decorate,decoration={markings,mark=at position 0.5 with {\arrow{>}}}}] plot [smooth, tension=1] coordinates {(2,3) (0.9,0.7) (4,1)};
        \node[blue,scale=1] at (2.5, 0.3) {$[1]\gamma$};
        
        \filldraw[black] (2,3) circle (2pt);
        \filldraw[black] (4,1) circle (2pt);
        \filldraw[black] (2,-1) circle (2pt);
        \filldraw[black] (0,1) circle (2pt);

        \filldraw[black] (1.3,1) circle (2pt);
        \filldraw[black] (2.7,1) circle (2pt);

        \draw[thick] (2,1) circle (2);
    \end{tikzpicture}
    \caption{An example of a one-end shift.}
    \label{fig:endshift}
\end{figure}

\begin{definition}\label{def-poliwhirl}
    Let $\gamma_0$ be the (untagged) straight line connecting $P_1$ and $P_2$.
    If $\gamma(0),\gamma(1) \in \PP$, and $\gamma$ without tagging is not homotopic to $\gamma_0$, then $\gamma\vartheta$ is the \textbf{right poliwhirl} of $\gamma$ obtained in the following manner:
    \begin{enumerate}
        \item homotope $\gamma(1)$ along $\gamma_0$ to the other puncture.
    \end{enumerate}
    If $\gamma(0),\gamma(1) \in \PP$, then $\vartheta\gamma$ is the \textbf{left poliwhirl} of $\gamma$ obtained in the following manner:
    \begin{enumerate}
        \item let $\mathbf{C}$ be the closed loop of radius $\varepsilon$ around $\gamma(1) \in \PP$
        \item homotope $\gamma(1)$ along $\gamma_0$ toward the other puncture until you reach the intersection of $\mathbf{C}$ and $\gamma_0$
        \item homotope $\gamma(1)$ along $\mathbf{C}$ \textit{once} clockwise until you are back at the intersection of $\mathbf{C}$ and $\gamma_0$
        \item continue the homotopy of step (2) along $\gamma_0$ to the other puncture.
    \end{enumerate}
\end{definition}

\begin{remark}
    Figure~\ref{fig:poliwhirl} gives an example of the left poliwhirl, $\gamma_0$, and the curve $\mathbf{C}$.
    Even though it is not obvious at first glance, $\vartheta\gamma\vartheta$ is homotopic to $\gamma$.
    The easy way to distinguish these operations is that the left poliwhirl $\vartheta\gamma$ has a clockwise (left-handed) movement, and $\vartheta$ is written on the left.
    The poliwhirl operation is named after the Pok\'emon of the same name, which has a spiral on its chest.
    When drawing examples of these moves, the right poliwhirl looks like unwinding a spool of thread by one rotation, and the left poliwhirl looks like adding another rotation to a spool of thread, hence the name.
\end{remark}

The next definition is of great importance to the remainder of the paper. 
It is the geometric analog of the Auslander-Reiten translation $\tau$.

\begin{definition}\label{def-rho}
    Let $(\gamma, \kappa, \lambda)$ be an admissible edge in $\surf$. 
    The \textbf{tagged rotation} of $(\gamma,\kappa, \lambda)$ is $\rho(\gamma,\kappa, \lambda) = (\rho(\gamma),\kappa', \lambda)$ where 
    $$\rho(\gamma) = \begin{dcases} [1]\gamma[1] & \text{ if } \gamma(0), \gamma(1) \in \MM \\
    [1]\gamma & \text{ if } \gamma(0) \in \PP \text{ and } \gamma(1) \in \MM \\
    \gamma & \text{ if } \gamma(0), \gamma(1) \in \PP
    \end{dcases}$$ 
    and $$\kappa(t)' =\begin{dcases}
        \kappa(t) & \text{if } \lambda \in \{0, 1, \infty\} \\
        1 - \kappa(t) & \text{otherwise}.
    \end{dcases}$$ In other words, the tagging map remains unchanged in the tubes of rank 1.
\end{definition}

\begin{figure}[h]
    \begin{tikzpicture}
        \draw[blue,thick, postaction={decorate,decoration={markings,mark=at position 0.25 with {\arrow{>}}}}] plot [smooth, tension=1] coordinates {(0,1) (2,2.25) (5,1) (2,-0.25) (0,1)};
        \node[blue,scale=1] at (2, -0.6) {$\vartheta\gamma$};
        \draw[blue,thick, postaction={decorate,decoration={markings,mark=at position 0.5 with {\arrow{>}}}}] plot [smooth, tension=1] coordinates {(0,1) (2,0.7) (4,1)};
        \node[blue,scale=1] at (2.5, 0.3) {$\gamma$};
        \draw[magenta,thick,dotted] plot [smooth, tension=1] coordinates {(4,1) (0,1)};
        \node[magenta,scale=1] at (2, 1.3) {$\gamma_0$};
        \draw[magenta,thick,dotted] (4,1) circle (0.5);
        \node[magenta,scale=1] at (3.3, 1.5) {$\mathbf{C}$};
        
        \filldraw[black] (4,1) circle (2pt);
        \filldraw[black] (0,1) circle (2pt);

    \end{tikzpicture}
    \caption{The poliwhirl in action.}
    \label{fig:poliwhirl}
\end{figure}
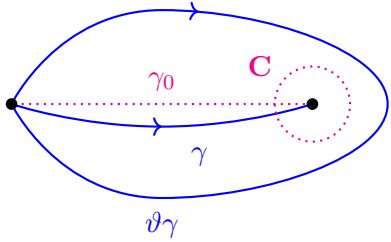

\begin{remark}
    Note that $\rho$ preserves adjacency between admissible edges and the coloring of edges.
    Specifically, if $(\gamma_1,\kappa_1,\lambda_1)$ and $(\gamma_2,\kappa_2,\lambda_2)$ share any endpoints, so do $\rho(\gamma_1,\kappa_1, \lambda_1)$ and $\rho(\gamma_2,\kappa_2, \lambda_2)$.
    Any intersections (including self-intersections) are invariant under $\rho$.
    As a consequence, if $\TT$ is an acyclic triangulation with quiver $Q^\TT$, then $\rho\TT$ is also an acyclic triangulation and $Q^{\rho\TT} = Q^\TT$.
    When we apply $n$ tagged rotations to a curve $(\gamma, \kappa, \lambda)$, we write $\rho^n(\gamma, \kappa, \lambda)$, so $\rho^0(\gamma, \kappa, \lambda) = (\gamma, \kappa, \lambda)$.
    Also, we don't mess with the tagging if the color of the admissible edge is an element of $\mathbb{P}^1$ that is not $0, 1, \infty$; this is just to ensure $\rho(\gamma, \kappa, \lambda)$ is homotopic to $(\gamma, \kappa, \lambda)$ for curves in stable tubes of rank 1.
\end{remark}

The requirement that no admissible edges cut out a once-punctured monogon is now an issue that must be considered in light of the $[1]$ operator.
In order to address this, He, Zhou, and Zhu \cite{he_geometric_2023} introduce the notion of the completion of a curve, which is a formalization of the process described in Fomin, Shapiro, Thurston \cite{fomin_cluster_2008}.

\begin{definition}\label{def-completion}
    If $\gamma(0) \in \PP$ and $\gamma(1) \in \MM$, then $\overline\gamma$, the \textbf{completion of $\gamma$}, is the curve which cuts out a once-punctured monogon and, together with $\gamma$, creates a self-folded triangle.
    Otherwise, $\overline\gamma = \gamma$.
    Note that there are two tagged edges ending in the puncture with the same completion; one for each map $\kappa$. 
    To avoid ambiguity, assume that all completions are oriented so that $\gamma(0) \in \PP$ lies to the left of $\overline\gamma$ (see Figure~\ref{fig:completion}).
\end{definition}

\begin{figure}[h]
    \begin{tikzpicture}
        \draw[blue,thick, postaction={decorate,decoration={markings,mark=at position 0.5 with {\arrow{>}}}}] plot [smooth, tension=1] coordinates {(2,1) (2,-1)};
        \node[blue,scale=1] at (1.8, 0.7) {$\gamma$};
        \draw[blue,thick, postaction={decorate,decoration={markings,mark=at position 0.5 with {\arrow{>}}}}] plot [smooth, tension=1] coordinates {(2,-1) (3,0.7) (2,1.5) (1,0.7) (2,-1)};
        \node[blue,scale=1] at (2, 2) {$\overline{\gamma}$};
        
        \filldraw[black] (2,3) circle (2pt);
        \filldraw[black] (4,1) circle (2pt);
        \filldraw[black] (2,-1) circle (2pt);
        \filldraw[black] (0,1) circle (2pt);

        \filldraw[black] (2,1) circle (2pt);

        \draw[thick] (2,1) circle (2);
    \end{tikzpicture}
    \caption{The completion of a tagged edge.}
    \label{fig:completion}
\end{figure}
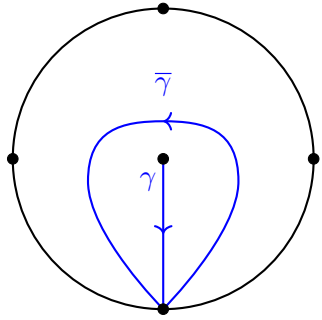

\section{The Category and Quiver of Admissible Edges}

This is the first of the two categories that Theorem~\ref{main-thm} is concerned with.
In \cite{fomin_cluster_2008}, the authors give a precise count for the number of tagged edges in a punctured, marked surface with boundary $\surf$.
The following theorem is due to Proposition~2.10 of their paper.

\begin{theorem}[\cite{fomin_cluster_2008}] \label{thm-nplus3}
    All triangulations $\TT$ of a twice-punctured $(n-2)$-gon contain $n+1$ admissible edges and have a quiver $Q^\TT$ of type $\widetilde{D}_{n}$.
\end{theorem}

We focus on triangulations of a twice-punctured disk with $(n-2)$ marked points ($n \geq 4$).
We will continue to assume that $\TT$ is an acyclic triangulation.

\begin{definition}
    Fix a triangulation $\TT$ of $\surf$.
    The \textbf{projective edges} in $E^\times_\lambda$ are defined to be 
    $$PTE(\TT) \coloneqq \{ (\gamma, \kappa, -\infty) \in \surf \ | \ \rho^1(\gamma, \kappa, -\infty) \in \TT \} \subset E^\times_\lambda.$$
    The \textbf{preprojective edges} in $E^\times_\lambda$ are defined to be
    $$\mathcal{P}(\TT) \coloneqq \{ (\gamma, \kappa, -\infty) \in \surf \ | \ \rho^i(\gamma, \kappa, -\infty) \in \TT \text{ for some } i \geq 1 \} \subset E^\times_\lambda.$$
    The \textbf{injective edges} in $E^\times_\lambda$ are defined to be 
    $$ITE(\TT) \coloneqq \{ (\gamma, \kappa, -\infty) \in \surf \ | \ \rho^{-1}(\gamma, \kappa, -\infty) \in \TT \} \subset E^\times_\lambda.$$
    The \textbf{preinjective edges} in $E^\times_\lambda$ are defined to be
    $$\mathcal{Q}(\TT) \coloneqq \{ (\gamma, \kappa, -\infty) \in \surf \ | \ \rho^{-i}(\gamma, \kappa, -\infty) \in \TT \text{ for some } i \geq 1 \} \subset E^\times_\lambda.$$
    The \textbf{regular edges} in $E^\times_\lambda$ are defined to be
    $$\mathcal{R}(\TT) \coloneqq \{ (\gamma, \kappa, \lambda) \in \surf \ | \ \rho^{i}(\gamma, \kappa, \lambda) = (\gamma, \kappa, \lambda) \text{ for some } i \geq 1 \} \subset E^\times_\lambda.$$
\end{definition}

\begin{remark}
    By Definition~\ref{def-tri}, all edges in a triangulation have color $-\infty$; this motivates the coloring of the preprojective and preinjective edges in the above definition.
\end{remark}

Since $\mathcal{P}(\TT)$ is dual to $\mathcal{Q}(\TT)$, it will be enough to prove our theorems for preprojective and regular edges only.
Also, in type $\widetilde{D}_n$, regular edges will have an orbit of size $(n-2)$, $2$, or $1$.
The edges in $\mathcal{R}(\TT)$ with an orbit of size $2$ and $1$ are the edges with both endpoints in the punctures, and the edges in $\mathcal{R}(\TT)$ with an orbit of size $(n-2)$ are the admissible edges that cut out a twice-punctured disk.
This aligns with what we observed in Example~\ref{running-exa}.

\subsection{Elementary Moves}

\begin{figure}[h]
    \centering
    \begin{tikzpicture}
    \begin{scope}[shift={(-5,0)}] 
        \draw[blue,thick] plot [smooth, tension=1] coordinates {(0,-2) (-1,0.7) (-2,0)};
        \node at (-0.58,0) {
          \tikz[rotate=-15,scale=1]{
            \draw[blue,->,thick] (0,0) -- (-0.1,0.1);
          }
        };
        \node[blue,scale=1] at (-0.7, 1) {$[1]\gamma$};
        \draw[blue,thick, postaction={decorate,decoration={markings,mark=at position 0.5 with {\arrow{>}}}}] plot coordinates {(0,-2) (0,2)};
        \node[blue,scale=1] at (0.25, -1) {$\gamma$};
        \draw[blue,thick, postaction={decorate,decoration={markings,mark=at position 0.5 with {\arrow{>}}}}] plot [smooth, tension=1] coordinates {(2,0) (1,-0.7) (0,2)};
        \node[blue,scale=1] at (0.7, 1) {$\gamma[1]$};
        
        \filldraw[black] (0,2) circle (2pt);
        \filldraw[black] (2,0) circle (2pt);
        \filldraw[black] (0,-2) circle (2pt);
        \filldraw[black] (-2,0) circle (2pt);

        \filldraw[black] (1,0) circle (2pt);
        \filldraw[black] (-1,0) circle (2pt);

        \draw[thick] (0,0) circle (2);
    \end{scope}
    \begin{scope}[shift={(0,0)}] 
        \draw[blue,thick] plot [smooth, tension=1] coordinates {(-1,0) (-0.2,-1) (0,-2)};
        \node[rotate=45, text=blue, scale=0.8] at (-0.8,-0.2) {$\bowtie$};
        \node at (-0.4,-0.675) {
          \tikz[rotate=-195,scale=1]{
            \draw[blue,->,thick] (0,0) -- (-0.1,0.1);
          }
        };
        \draw[blue,thick, postaction={decorate,decoration={markings,mark=at position 0.5 with {\arrow{>}}}}] plot [smooth, tension=1] coordinates {(-1,0) (-0.9,-1) (0,-2)};
        \node[blue,scale=1] at (-0.6, -1) {$\eta$};
        \draw[olive,thick, postaction={decorate,decoration={markings,mark=at position 0.5 with {\arrow{>}}}}] plot [smooth, tension=1] coordinates {(0,-2) (-0.2,0) (-1.1,0.4) (-1.5,-0.6) (0,-2)};
        \node[olive,scale=1] at (-1.4, -0.2) {$\overline{\eta}$};
        \draw[blue,thick] plot [smooth, tension=1] coordinates {(0,-2) (0,1) (-2,0)};
        \node at (0.18,0.7) {
          \tikz[rotate=-20,scale=1]{
            \draw[blue,->,thick] (0,0) -- (-0.1,0.1);
          }
        };
        \node[blue,scale=1] at (-0.3, 0.7) {$\gamma$};
        \draw[blue,thick, postaction={decorate,decoration={markings,mark=at position 0.5 with {\arrow{>}}}}] plot [smooth, tension=1] coordinates {(2,0) (0,1.5) (-2,0)};
        \node[blue,scale=1] at (1, 0.5) {$\gamma[1]$};
        
        \filldraw[black] (0,2) circle (2pt);
        \filldraw[black] (2,0) circle (2pt);
        \filldraw[black] (0,-2) circle (2pt);
        \filldraw[black] (-2,0) circle (2pt);

        \filldraw[black] (1,0) circle (2pt);
        \filldraw[black] (-1,0) circle (2pt);

        \draw[thick] (0,0) circle (2);
    \end{scope}
    \begin{scope}[shift={(5,0)}] 
        \draw[blue,thick, postaction={decorate,decoration={markings,mark=at position 0.5 with {\arrow{>}}}}] plot coordinates {(-1,0) (-2,0)};
        \node[blue,scale=1] at (-1.5, -0.3) {$\gamma$};
        \draw[blue,thick] plot [smooth, tension=1] coordinates {(0,-2) (0,0.7) (-2,0)};
        \node at (0.16,0.45) {
          \tikz[rotate=-20,scale=1]{
            \draw[blue,->,thick] (0,0) -- (-0.1,0.1);
          }
        };
        \node[blue,scale=1] at (0.6, 0.7) {$\overline{\gamma}[1]$};
        
        \filldraw[black] (0,2) circle (2pt);
        \filldraw[black] (2,0) circle (2pt);
        \filldraw[black] (0,-2) circle (2pt);
        \filldraw[black] (-2,0) circle (2pt);

        \filldraw[black] (1,0) circle (2pt);
        \filldraw[black] (-1,0) circle (2pt);

        \draw[thick] (0,0) circle (2);
    \end{scope}

    \begin{scope}[shift={(-5,-5)}] 
        \draw[blue,thick, postaction={decorate,decoration={markings,mark=at position 0.75 with {\arrow{>}}},decoration={markings,mark=at position 0.15 with {\node[rotate=120, text=blue, scale=0.8] at (0,-0.25) {$\bowtie$};}}}] plot [smooth, tension=1] coordinates {(-1,0) (0,0.6) (1.5,0) (0,-0.6) (-1,0)};
        \node[blue,scale=1] at (-0.3, -0.8) {$\gamma$};
        \draw[blue,thick, postaction={decorate,decoration={markings,mark=at position 0.75 with {\arrow{>}}},decoration={markings,mark=at position 0.2 with {\node[rotate=90, text=blue, scale=0.8] at (0,0) {$\bowtie$};}}}] plot [smooth, tension=1] coordinates {(-1,0) (1,0)};
        \node[blue,scale=1] at (0.2, 0.25) {$\gamma\vartheta$};
        \draw[blue,thick, postaction={decorate,decoration={markings,mark=at position 0.5 with {\arrow{>}}},decoration={markings,mark=at position 0.98 with {\node[rotate=30, text=blue, scale=0.8] at (0,0) {$\bowtie$};}}}] plot [smooth, tension=1] coordinates {(-1,0) (0.5,1.6) (1.75,0) (0,-1.6) (-1.75,0) (-0.5,1.6) (1,0)};
        \node[blue,scale=1] at (0.1, -1.3) {$\vartheta\gamma$};
        
        \filldraw[black] (0,2) circle (2pt);
        \filldraw[black] (2,0) circle (2pt);
        \filldraw[black] (0,-2) circle (2pt);
        \filldraw[black] (-2,0) circle (2pt);

        \filldraw[black] (1,0) circle (2pt);
        \filldraw[black] (-1,0) circle (2pt);

        \draw[thick] (0,0) circle (2);
    \end{scope}
    \begin{scope}[shift={(0,-5)}] 
        \draw[blue,thick, postaction={decorate,decoration={markings,mark=at position 0.5 with {\arrow{>}}}}] plot [smooth, tension=1] coordinates {(0,-2) (-1.75,0) (0,2)};
        \node[blue,scale=1] at (0, 1.5) {$[1]\gamma$};
        \draw[blue,thick, postaction={decorate,decoration={markings,mark=at position 0.5 with {\arrow{>}}}}] plot [smooth, tension=1] coordinates {(0,-2) (-1.5,0) (0,1) (2,0)};
        \node[blue,scale=1] at (1, 1) {$\gamma$};
        \draw[blue,thick, postaction={decorate,decoration={markings,mark=at position 0.5 with {\arrow{>}}}}] plot [smooth, tension=1] coordinates {(0,-2) (-1.3,-0.2) (0,0.7) (1.3,-0.2) (0,-2)};
        \node[blue,scale=1] at (0, 0.3) {$\gamma[1]$};
        
        \filldraw[black] (0,2) circle (2pt);
        \filldraw[black] (2,0) circle (2pt);
        \filldraw[black] (0,-2) circle (2pt);
        \filldraw[black] (-2,0) circle (2pt);

        \filldraw[black] (1,0) circle (2pt);
        \filldraw[black] (-1,0) circle (2pt);

        \draw[thick] (0,0) circle (2);
    \end{scope}
    \begin{scope}[shift={(5,-5)}] 
        \draw[blue,thick, postaction={decorate,decoration={markings,mark=at position 0.75 with {\arrow{>}}},decoration={markings,mark=at position 0.85 with {\node[rotate=60, text=blue, scale=0.8] at (0,0) {$\bowtie$};}}}] plot [smooth, tension=1] coordinates {(-1,0) (0,0.6) (1.2,0) (0,-0.6) (-1,0)};
        \node[blue,scale=1] at (0, -0.3) {$\gamma$};
        \draw[blue,thick, postaction={decorate,decoration={markings,mark=at position 0.88 with {\arrow{>}}},decoration={markings,mark=at position 0.97 with {\node[rotate=20, text=blue, scale=0.8] at (0,0) {$\bowtie$};}}}] plot [smooth, tension=1] coordinates {(-1,0) (0,1) (1.5,0) (0,-1) (-1.5,0) (0,1.4) (1.7,0) (0,-1.4) (-1,0)};
        \node[blue,scale=1] at (0, -1.7) {$\vartheta^2\gamma$};
        
        \filldraw[black] (0,2) circle (2pt);
        \filldraw[black] (2,0) circle (2pt);
        \filldraw[black] (0,-2) circle (2pt);
        \filldraw[black] (-2,0) circle (2pt);

        \filldraw[black] (1,0) circle (2pt);
        \filldraw[black] (-1,0) circle (2pt);

        \draw[thick] (0,0) circle (2);
    \end{scope}
    \end{tikzpicture}
    \caption{The six classes of elementary moves left-to-right, top-to-bottom.}
    \label{fig:elementary}
\end{figure}
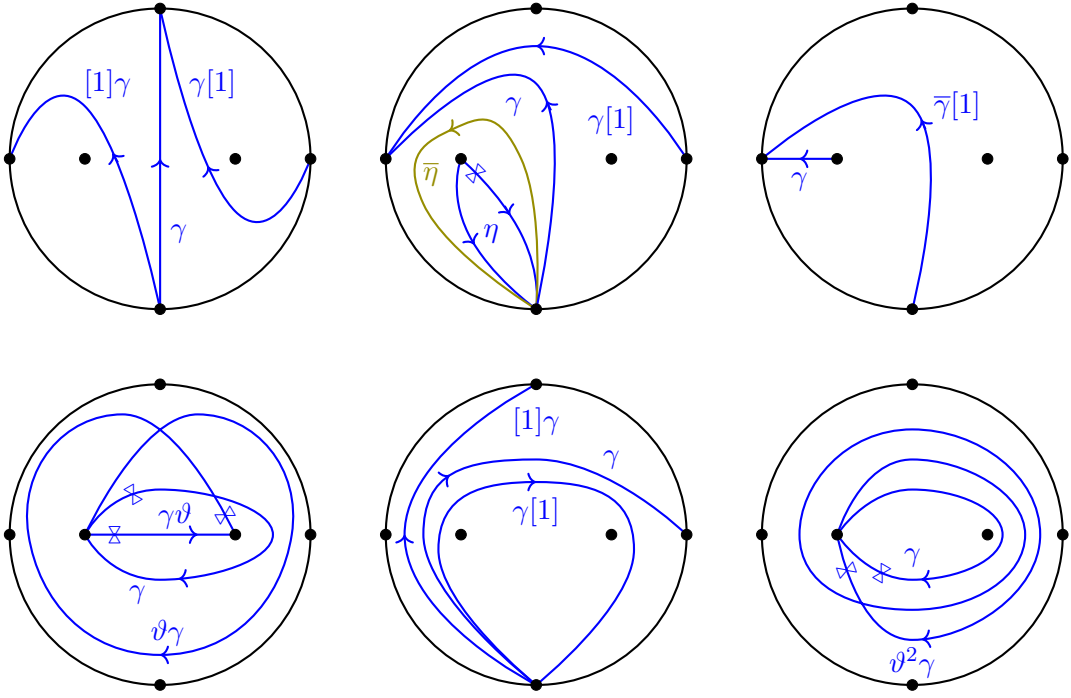

An elementary move sends an admissible edge $( \gamma, \kappa, \lambda )$ to another admissible edge $(\eta, \kappa', \lambda)$.
Conceptually, these are all the possible intermediate steps between $(\gamma, \kappa, \lambda)$ and $\rho(\gamma, \kappa, \lambda)$.
The irreducible morphisms in $\mod \k Q$ are the categorification of the inverses of elementary moves. 
Ultimately, the $(\eta, \kappa', \lambda)$ will form the geometric analog of the middle terms in an almost split sequence ending at $( \gamma, \kappa, \lambda)$.

\begin{definition}
    Let $\surf$ be a twice-punctured $(n-2)$-gon, $\TT$ be an acyclic triangulation of $\surf$, $\gamma_0$ be as above, and $(\gamma,\kappa,\lambda) \in E^\times_\lambda \setminus PTE(\TT)$ be a non-projective admissible edge. 
    An \textbf{elementary move} is a map $E^\times_\lambda \setminus PTE(\TT) \to E^\times_\lambda$ which falls into one of six cases:
    \begin{enumerate}
        \item If $\gamma$ has both ends in $\MM$, $\gamma$ separates $\surf$ into two once-punctured polygons (i.e., $\gamma$ intersects $\gamma_0$ transversely), and $\gamma$ and the boundary components of $\surf$ \textbf{do not} form a once-punctured digon, then there are exactly two elementary moves $( \gamma, \emptyset, -\infty ) \mapsto ( \gamma[1], \emptyset, -\infty )$ and $( \gamma, \emptyset, -\infty ) \mapsto ( [1]\gamma, \emptyset, -\infty ).$
        \item If $\gamma$ has both ends in $\MM$, $\gamma$ separates $\surf$ into two once-punctured polygons (i.e., $\gamma$ intersects $\gamma_0$ transversely), and $\gamma$ and a boundary component of $\surf$ \textbf{do} form a once-punctured digon, then there are exactly three elementary moves. 
        Without loss of generality, assume that $\gamma$ is oriented so that the puncture inside the digon lies to the left of $\gamma$. 
        In this case, $[1]\gamma$ is homotopic to the completion $\overline{\eta}$ of two tagged edges $(\eta,\kappa)$ and $(\eta,\kappa')$. 
        The elementary moves are therefore $( \gamma, \emptyset, -\infty ) \mapsto ( \gamma[1], \emptyset, -\infty )$, $( \gamma, \emptyset, -\infty ) \mapsto ( \eta, \kappa, -\infty)$, and $( \gamma, \emptyset, -\infty ) \mapsto ( \eta, 1-\kappa, -\infty).$
        \item If $\gamma$ satisfies $\gamma(0) \in \PP$ and $\gamma(1) \in \MM$, then there is exactly one elementary move. Let $\overline{\gamma}$ be the completion of $\gamma$. The elementary move is $( \gamma, \kappa, -\infty ) \mapsto ( \overline{\gamma}[1], \emptyset, -\infty ).$ 
        \item If $\gamma$ satisfies $\gamma(0), \gamma(1) \in \PP$ and $\rho^2(\gamma) = \gamma$, then there are exactly two elementary moves $( \gamma, \kappa, \lambda ) \mapsto ( \vartheta\gamma, 1-\kappa, \lambda )$ and $( \gamma, \kappa, \lambda ) \mapsto ( \gamma\vartheta, \kappa, \lambda ).$\footnote{The last elementary move is only allowed if $\gamma\vartheta$ is not nullhomotopic.}
        In this case, $\lambda \in \{ 0,1\}.$
        \item If $\gamma$ has both ends in $\MM$ and $\gamma$ does not intersects $\gamma_0$, then there are exactly two elementary moves $( \gamma, \emptyset, \infty ) \mapsto ( \gamma[1], \emptyset, \infty )$ and $( \gamma, \emptyset, \infty ) \mapsto ( [1]\gamma, \emptyset, \infty ).$\footnote{The last elementary move is only allowed if $[1]\gamma$ is not homotopic to a boundary component.}
        \item If $\gamma$ satisfies $\gamma(0), \gamma(1) \in \PP$ and $\rho^1(\gamma) = \gamma$, then there are exactly two elementary moves $( \gamma, \kappa, \lambda ) \mapsto ( \vartheta^2\gamma, \kappa, \lambda )$ and $( \gamma, \kappa, \lambda ) \mapsto ( \gamma\vartheta^2, \kappa, \lambda ).$\footnote{The last elementary move is only allowed if $(\Int((\gamma,\kappa, \lambda), i ))_{i \in \TT} \neq (1,1,2,2,\ldots,2,1,1) = \delta.$}
    \end{enumerate}
\end{definition}

\begin{remark}
    The elementary moves described above can be thought of as possible ``middle steps'' between a curve $(\gamma,\kappa, \lambda)$ and its tagged rotation $\rho(\gamma,\kappa, \lambda)$.
    Moves (1), (2), and (3) are applicable only to preprojective and preinjective edges, hence the coloring of $-\infty$.
    Move (4) is applicable to admissible edges in stable tubes of rank $2$, hence the coloring of $0$ or $1$.
    Move (5) is applicable to admissible edges in the stable tube of rank $(n-2)$, hence the coloring of $\infty$.
    Move (6) is applicable to admissible edges in stable tubes of rank $1$.
\end{remark}

\subsection{The Category of Admissible Edges}\label{sec-edge-category}
Now that $E^\times_\lambda$ has been described and we have the notion of an elementary move, we can define the second category that plays a major role in Theorem~\ref{main-thm}.

\begin{definition}
    A \textbf{mesh relation} is an equality between certain sequences of elementary moves.
    To be precise, for $(\gamma, \kappa,\lambda) \in E^\times_\lambda$, the mesh relation is \[m_{(\gamma,\kappa,\lambda)} = \sum_\alpha (\alpha)\alpha\]
    where the sum is over all elementary moves $\alpha: (\gamma,\kappa,\lambda) \mapsto (\lambda_i, \kappa_i,\lambda)$ and $(\alpha)$ is the elementary move $(\alpha): (\lambda_i, \kappa_i,\lambda) \mapsto \rho(\gamma,\kappa,\lambda)$ which is guaranteed to exist for $(\gamma, \kappa,\lambda) \notin PTE(\TT)$.
\end{definition}
\begin{definition}
    The \textbf{category of admissible edges} $\mathcal{E(\TT)}$ is the $\k$-linear additive category with objects being direct sums of endpoint-relative homotopy classes of admissible edges.
    Since this is an additive category, we need only define morphisms between admissible edges.
    The space of morphisms from $(\gamma_1, \kappa_1, \lambda_1)$ to $(\gamma_2, \kappa_2, \lambda_2)$ is the quotient of the vector space over $\k$ spanned by sequences of elementary moves from $(\gamma_1, \kappa_1, \lambda_1) $ to $ (\gamma_2, \kappa_2, \lambda_2)$ by the subspace generated by mesh relations.
\end{definition}

We now define the quiver of admissible edges. 

\begin{definition}
    The \textbf{quiver of admissible edges} $\Gamma(\mathcal{E(\TT)})$ is the quiver with vertices the endpoint-relative homotopy classes of admissible edges in $\mathcal{E(\TT)}$ and arrows $(\gamma_2, \kappa_2, \lambda_2) \to (\gamma_1, \kappa_1, \lambda_1)$ whenever there is an elementary move $(\gamma_1, \kappa_1, \lambda_1) \mapsto (\gamma_2, \kappa_2, \lambda_2)$.
\end{definition}

\begin{remark}
    It is crucial to note that the arrows in $\Gamma(\mathcal{E(\TT)})$ go in the \textit{opposite} direction of the elementary moves.
    Also, note that the quiver of admissible edges will have 3 components: the preprojective component $\mathcal{P(\TT)}$ with a definite ``beginning'' starting with the projective tagged edges, the preinjective component $\mathcal{Q(\TT)}$ with a definite ``end'' consisting of the injective tagged edges, and the regular component $\mathcal{R(\TT)}$ which itself is made of infinitely many stable tubes.
    A pictorial representation of the Auslander-Reiten quiver of the module category can be found in Figure~\ref{fig:modules category}, and the same picture works for describing the quiver of admissible edges.
\end{remark}

It shouldn't be much of a surprise that $\rho$ can be thought of as a right-to-left translation in $\Gamma(\mathcal{E(\TT)})$. 
The following lemma makes clear the structure of $\mathcal{R}(\TT)$.

\begin{lemma}
    $\mathcal{R(\TT)}$ is made up of connected components called stable tubes.
    The tubes are made of admissible edges that have a finite orbit under $\rho$.
    There are two stable tubes of rank 2, one of rank $(n-2)$, and infinitely many of rank 1.
    The stable tubes of rank 2 are composed of admissible edges with both endpoints in the punctures, i.e., $\gamma(0), \gamma(1) \in \PP$. 
    The difference between the edges in each of the tubes of rank 2 is the tagging (and coloring); in one stable tube of rank 2, we have $\kappa(0) = \kappa(1)$ for all admissible edges, and in the other we have $\kappa(0) \neq \kappa(1)$.
    The edges in the tube of rank $(n-2)$ cut out a twice-punctured polygon, and the rank of this tube is the same as the number of marked points in $\MM$.
    Finally, there are infinitely many stable tubes of rank 1 that are made up of edges with both endpoints in the punctures and who are stable under $\rho$.
\end{lemma}

\section{An Equivalence of Quivers}

\begin{definition}\label{def-phi}
    Let $\TT$ be an acyclic triangulation of a twice-punctured disk with $(n-2)$ marked boundary points corresponding to an acyclic quiver $Q^\TT$ of type $\widetilde{D}_n$.
    Define an isomorphism of quivers $\varphi: \Gamma(\mod\k Q^\TT) \to \Gamma(\mathcal{E}(\TT))$ in the following way:
    \begin{itemize}
        \item $\varphi$ sends projective indecomposable modules at vertex to projective admissible edges and injective indecomposable modules to injective admissible edges (with respect to the labeling).
        \item If $[M]$ is an isomorphism class of indecomposable modules at the mouth of the stable tube $\mathcal{T}^\lambda$, then $[M]$ is uniquely determined by it's dimension vector and the number $\lambda$.
        Define $\varphi([M]) = (\gamma, \kappa, \lambda)$ where $\udim [M] = (\Int((\gamma,\kappa, \lambda), i ))_{i \in \TT}$ and $(\gamma, \kappa, \lambda)$ also sits at the mouth of a stable tube.
        \item We define $\varphi$ on the remaining vertices of the Auslander-Reiten quiver recursively: if there is an arrow $[M] \to [N]$ in $\Gamma(\mod\k Q^\TT)$, an arrow $\varphi([M]) \to (\gamma, \kappa, \lambda)$ in $\Gamma(\mathcal{E}(\TT))$, and $\udim [N] = (\Int((\gamma,\kappa, \lambda), i ))_{i \in \TT}$, then define $\varphi([N]) = (\gamma,\kappa, \lambda).$ 
        Dually, if there is an arrow $[M] \to [N]$ in $\Gamma(\mod\k Q^\TT)$, an arrow $(\gamma, \kappa, \lambda) \to \varphi([N])$ in $\Gamma(\mathcal{E}(\TT))$, and $\udim [M] = (\Int((\gamma,\kappa, \lambda), i ))_{i \in \TT}$, then define $\varphi([M]) = (\gamma,\kappa, \lambda).$ 
    \end{itemize}
\end{definition}

It should be clear from the constructions in the previous sections and the definition that $\varphi$ is an isomorphism of quivers. 
The recursive definition of $\varphi$ ensures that arrows are mapped to arrows, and since each connected component of $\Gamma(\mod\k Q^\TT)$ and $\Gamma(\mathcal{E}(\TT))$ is uniquely determined by either the projective, injective, or mouth vertices, Definition~\ref{def-phi} must define an isomorphism of quivers. 
Moreover, this isomorphism has the following property: $\varphi \circ \tau = \rho \circ \varphi$.

\section{The Intersection-Dimension Formula}

\begin{lemma}
    Let $(\gamma_1, \kappa_1, -\infty) \in \mathcal{P}(\TT)$ and $(\gamma_2, \kappa_2, \lambda) \in E^\times_\lambda$ where $\TT$ is an acyclic triangulation of a twice-punctured $(n-2)$-gon $\surf$ and let $M(\gamma_1, \kappa_1, -\infty) = M_1 \in \mathcal{P}(\k Q^\TT)$ and $M(\gamma_2, \kappa_2, \lambda_2) = M_2 \in \Gamma(\mod \k Q^\TT)$ be the corresponding modules under the equivalence $\varphi^{-1}$.
    Then $$\Int( (\gamma_1,\kappa_1, -\infty), (\gamma_2,\kappa_2, \lambda) ) = \dim_\k \Ext^1 (M_1, M_2) + \dim_\k \Ext^1 (M_2, M_1).$$
\end{lemma}

\begin{proof}
    By Theorem~\ref{thm-0-homs}, we know that any nonzero homomorphism between $M_1$ and $M_2$ must begin in the preprojective component.
    If $M_2 \in \mathcal{P}(\k Q^\TT)$ as well, assume without loss of generality that $M_1$ is to the left of $M_2$ in the Auslander-Reiten quiver.
    First note that, for an indecomposable modules $M$, we have \[\Hom (P(j), M) \cong M_j\] where $M_j$ is the vector space at vertex $j$ in the indecomposable quiver representation $M$.
    Since $\Int( -, - )$, $\Hom(-,-)$, and $\Ext^1 (-, -)$ are invariant under $\tau$ and $\rho$ for non-projective modules over a path algebra of type $\widetilde{D}_n$, we can apply these translations to both objects until $(\gamma_1, \kappa_1, -\infty) \in \TT$.
    Say $\rho^{n+1} (\gamma_1,\kappa_1, -\infty) = j \in \TT$ so that $\tau^{n} M_1 = P(j)$.
    Then, if $e_j$ is the unit vector with a $1$ in the $j^{th}$ position and zeros elsewhere and $\tau^{n+1} M_2 \neq 0$,
    \begin{align*}
        \Int( (\gamma_1,\kappa_1, -\infty), (\gamma_2,\kappa_2, \lambda) ) &= \Int( \rho^{n+1}(\gamma_1,\kappa_1, -\infty), \rho^{n+1}(\gamma_2,\kappa_2, \lambda) ) \\
        &= \Int( j, \rho^{n+1}(\gamma_2,\kappa_2, \lambda) ) \\
        &= \udim (\tau^{n+1} M_2) \cdot e_j = (\tau^{n+1} M_2)_j \\
        &= \dim_\k \Hom (P(j), \tau^{n+1} M_2) \\
        &= \dim_\k \Hom (\tau^{-n}P(j), \tau^{-n}\tau^{n+1} M_2) \\
        &= \dim_\k \Hom (M_1, \tau M_2) \\
        &= \dim_\k \Ext^1 (M_2, M_1)
    \end{align*}
    where the last inequality follows from the Auslander-Reiten formulas.
    Moreover, by our assumption that $M_1$ is to the left of $M_2$, we have that $\Ext^1 (M_1, M_2) = 0$.
\end{proof}

\begin{lemma}
    Let $(\gamma_1, \kappa_1, \lambda_1), (\gamma_2, \kappa_2, \lambda_2) \in \mathcal{R}(\TT)$ where $\TT$ is an acyclic triangulation of a twice-punctured $(n-2)$-gon $\surf$ and let $M(\gamma_1, \kappa_1, \lambda_1) = M_1, M(\gamma_2, \kappa_2, \lambda_2) = M_2 \in \mathcal{R}(\k Q^\TT)$ be the corresponding modules under the equivalence $\varphi^{-1}$.
    Then $$\Int( (\gamma_1,\kappa_1, \lambda_1), (\gamma_2,\kappa_2, \lambda_2) ) = \dim_\k \Ext^1 (M_1, M_2) + \dim_\k \Ext^1 (M_2, M_1).$$
\end{lemma}

\begin{proof}
    If $\lambda_1 \neq \lambda_2$, then the intersection number is zero.
    Also, the two modules $M_1$ and $M_2$ live in disjoint stable tubes, and therefore there are no nonzero extensions between $M_1$ and $M_2$, so the right-hand side of the equation is also zero.

    If $\lambda_1 = \lambda_2$, then the two modules live in the same stable tube.
    Since each indecomposable object in a given stable tube is uniquely determined by its composition series (i.e., they are uniserial) with respect to the indecomposable modules at the mouth of the tube, calculation of $\Ext^1$ is particularly easy (see \cite[Example 2.12]{simson_elements_2007-1}).
    In particular, $\dim_\k \Ext^1 (-, -)$ can be calculated completely combinatorially within a given stable tube.
    Also, by construction (and visual inspection), the intersection number between two admissible edges in the same stable tube follows exactly the same pattern as their corresponding modules.
\end{proof}

\begin{corollary}[Theorem A]
    Let $\surf$ be a triangulated, twice-punctured marked surface whose triangulation $\TT$ corresponds to an acyclic quiver $Q^\TT$ of type $\widetilde{D}_n$. 
    Then given any two admissible edges $(\gamma_1, \kappa_1, \lambda_1)$ and $(\gamma_2, \kappa_2, \lambda_2)$ (not necessarily distinct), 
    $$\Int((\gamma_1, \kappa_1, \lambda_1), (\gamma_2, \kappa_2, \lambda_2)) = \dim_\k \Ext^1 (M_1, M_2) + \dim_\k \Ext^1 (M_2, M_1)$$ where $M_i = M(\gamma_i,\kappa_i, \lambda_i)$ under the correspondence $\varphi^{-1}$ and $\Int$ is the intersection number between two admissible edges.
\end{corollary}

\begin{example}
    Let's continue our running example:
    
    
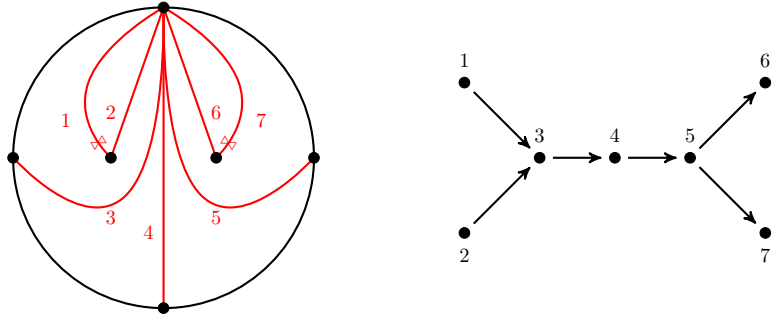
\begin{center}
        \begin{tikzpicture}
    \begin{scope}[shift={(0,0)}]

        \draw[red,thick, postaction={decorate,decoration={markings,mark=at position 0.9 with {\node[rotate=30, text=red, scale=0.6] at (0,0) {$\bowtie$};}}}] plot [smooth, tension=1] coordinates {(2,3) (1,2) (1.3,1)};
        \node[text=red, scale=0.7] at (0.7, 1.5) {1};
        \draw[red,thick] plot [smooth, tension=1] coordinates {(2,3) (1.3,1)};
        \node[text=red, scale=0.7] at (1.3, 1.6) {2};
        \draw[red,thick] plot [smooth, tension=1] coordinates {(2,3) (1.5,0.5) (0,1)};
        \node[text=red, scale=0.7] at (1.3, 0.2) {3};
        \draw[red,thick] plot [smooth, tension=1] coordinates {(2,3) (2,-1)};
        \node[text=red, scale=0.7] at (1.8, 0) {4};
        \draw[red,thick] plot [smooth, tension=1] coordinates {(2,3) (2.5,0.5) (4,1)};
        \node[text=red, scale=0.7] at (2.7, 0.2) {5};
        \draw[red,thick] plot [smooth, tension=1] coordinates {(2,3) (2.7,1)};
        \node[text=red, scale=0.7] at (2.7,1.6) {6};
        \draw[red,thick, postaction={decorate,decoration={markings,mark=at position 0.9 with {\node[rotate=-30, text=red, scale=0.6] at (0,0) {$\bowtie$};}}}] plot [smooth, tension=1] coordinates {(2,3) (3,2) (2.7,1)};
        \node[text=red, scale=0.7] at (3.3,1.5) {7};
        
        \filldraw[black] (2,3) circle (2pt);
        \filldraw[black] (4,1) circle (2pt);
        \filldraw[black] (2,-1) circle (2pt);
        \filldraw[black] (0,1) circle (2pt);

        \filldraw[black] (1.3,1) circle (2pt);
        \filldraw[black] (2.7,1) circle (2pt);

        \draw[thick] (2,1) circle (2);
    \end{scope}
    
      \begin{scope}[shift={(6,0)}]
        \filldraw[black] (0,0) circle (2pt);
        \node[scale=0.7] at (0, 2.3) {1};
        \filldraw[black] (0,2) circle (2pt);
        \node[scale=0.7] at (0, -0.3) {2};
        \filldraw[black] (1,1) circle (2pt);
        \node[scale=0.7] at (1, 1.3) {3};
        \filldraw[black] (2,1) circle (2pt);
        \node[scale=0.7] at (2, 1.3) {4};
        \filldraw[black] (3,1) circle (2pt);
        \node[scale=0.7] at (3, 1.3) {5};
        \filldraw[black] (4,2) circle (2pt);
        \node[scale=0.7] at (4, 2.3) {6};
        \filldraw[black] (4,0) circle (2pt);
        \node[scale=0.7] at (4, -0.3) {7};
        
        \draw[thick,->, shorten >=5, shorten <=5, >=stealth'](0,2)to(1,1);
        \draw[thick,->, shorten >=5, shorten <=5, >=stealth'](0,0)to(1,1);
        \draw[thick,->, shorten >=5, shorten <=5, >=stealth'](1,1)to(2,1);
        \draw[thick,->, shorten >=5, shorten <=5, >=stealth'](2,1)to(3,1);
        \draw[thick,->, shorten >=5, shorten <=5, >=stealth'](3,1)to(4,2);
        \draw[thick,->, shorten >=5, shorten <=5, >=stealth'](3,1)to(4,0);
      \end{scope}
    \end{tikzpicture}
    \end{center}    

    Recall that our three example modules $P(3), I(3),$ and $M_0$ had the following quiver representations (here $\lambda = 0$ in the last representation):

    
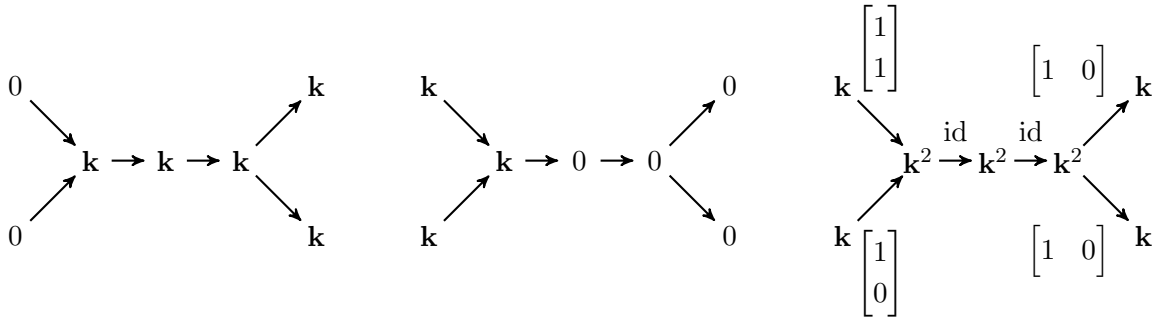
\begin{center}
    \begin{tikzpicture}
    \begin{scope}[shift={(0,0)}]
        \node at (0,2) {$0$};
        \node at (0,0) {$0$};
        \node at (1,1) {$\k$};
        \node at (2,1) {$\k$};
        \node at (3,1) {$\k$};
        \node at (4,2) {$\k$};
        \node at (4,0) {$\k$};
        
        \draw[thick,->, shorten >=8, shorten <=8, >=stealth'](0,2)to(1,1);
        \draw[thick,->, shorten >=8, shorten <=8, >=stealth'](0,0)to(1,1);
        \draw[thick,->, shorten >=8, shorten <=8, >=stealth'](1,1)to(2,1);
        \draw[thick,->, shorten >=8, shorten <=8, >=stealth'](2,1)to(3,1);
        \draw[thick,->, shorten >=8, shorten <=8, >=stealth'](3,1)to(4,2);
        \draw[thick,->, shorten >=8, shorten <=8, >=stealth'](3,1)to(4,0); 
    \end{scope}
    \begin{scope}[shift={(5.5,0)}]
        \node at (0,2) {$\k$};
        \node at (0,0) {$\k$};
        \node at (1,1) {$\k$};
        \node at (2,1) {$0$};
        \node at (3,1) {$0$};
        \node at (4,2) {$0$};
        \node at (4,0) {$0$};
        
        \draw[thick,->, shorten >=8, shorten <=8, >=stealth'](0,2)to(1,1);
        \draw[thick,->, shorten >=8, shorten <=8, >=stealth'](0,0)to(1,1);
        \draw[thick,->, shorten >=8, shorten <=8, >=stealth'](1,1)to(2,1);
        \draw[thick,->, shorten >=8, shorten <=8, >=stealth'](2,1)to(3,1);
        \draw[thick,->, shorten >=8, shorten <=8, >=stealth'](3,1)to(4,2);
        \draw[thick,->, shorten >=8, shorten <=8, >=stealth'](3,1)to(4,0); 
    \end{scope}
    \begin{scope}[shift={(11,0)}]
        \node at (0,2) {$\k$};
        \node at (0,0) {$\k$};
        \node at (1,1) {$\k^2$};
        \node at (2,1) {$\k^2$};
        \node at (3,1) {$\k^2$};
        \node at (4,2) {$\k$};
        \node at (4,0) {$\k$};

        \node at (0.5,2.5) {$\begin{bmatrix}
            1 \\
            1
        \end{bmatrix}$};
        \node at (0.5,-0.5) {$\begin{bmatrix}
            1 \\
            0
        \end{bmatrix}$};
        \node at (1.5,1.4) {id};
        \node at (2.5,1.4) {id};
        \node at (3,2.2) {$\begin{bmatrix}
            1 & 0
        \end{bmatrix}$};
        \node at (3,-0.2) {$\begin{bmatrix}
            1 & 0
        \end{bmatrix}$};
        
        \draw[thick,->, shorten >=8, shorten <=8, >=stealth'](0,2)to(1,1);
        \draw[thick,->, shorten >=8, shorten <=8, >=stealth'](0,0)to(1,1);
        \draw[thick,->, shorten >=8, shorten <=8, >=stealth'](1,1)to(2,1);
        \draw[thick,->, shorten >=8, shorten <=8, >=stealth'](2,1)to(3,1);
        \draw[thick,->, shorten >=8, shorten <=8, >=stealth'](3,1)to(4,2);
        \draw[thick,->, shorten >=8, shorten <=8, >=stealth'](3,1)to(4,0); 
    \end{scope}
    \end{tikzpicture}
    \end{center}

    Then, under our equivalence, we can visualize the corresponding edges: 

    
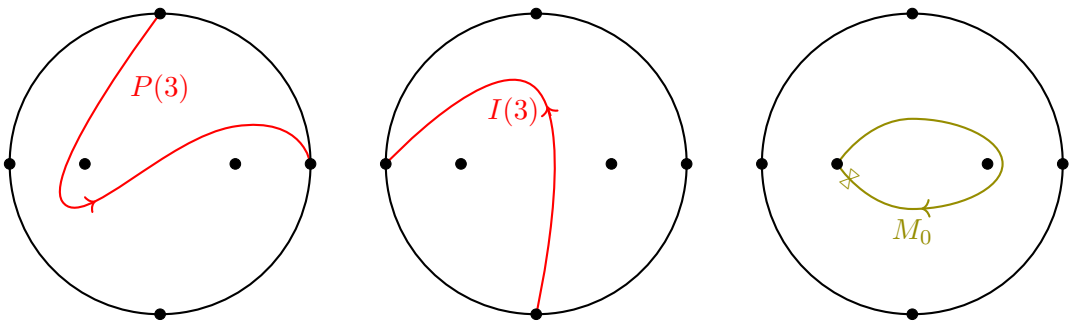
\begin{center}
        \begin{tikzpicture}
            \begin{scope}[shift={(-5,0)}] 
        \draw[red,thick, postaction={decorate,decoration={markings,mark=at position 0.5 with {\arrow{>}}}}] plot [smooth, tension=1] coordinates {(0,2) (-1.3,-0.5) (1,0.5) (2,0)};
        \node[red,scale=1] at (0, 1) {$P(3)$};
        
        \filldraw[black] (0,2) circle (2pt);
        \filldraw[black] (2,0) circle (2pt);
        \filldraw[black] (0,-2) circle (2pt);
        \filldraw[black] (-2,0) circle (2pt);

        \filldraw[black] (1,0) circle (2pt);
        \filldraw[black] (-1,0) circle (2pt);

        \draw[thick] (0,0) circle (2);
    \end{scope}
    \begin{scope}[shift={(0,0)}] 
        \draw[red,thick] plot [smooth, tension=1] coordinates {(0,-2) (0,1) (-2,0)};
        \node at (0.18,0.7) {
          \tikz[rotate=-25,scale=1]{
            \draw[red,->,thick] (0,0) -- (-0.1,0.1);
          }
        };
        \node[red,scale=1] at (-0.3, 0.7) {$I(3)$};
        
        \filldraw[black] (0,2) circle (2pt);
        \filldraw[black] (2,0) circle (2pt);
        \filldraw[black] (0,-2) circle (2pt);
        \filldraw[black] (-2,0) circle (2pt);

        \filldraw[black] (1,0) circle (2pt);
        \filldraw[black] (-1,0) circle (2pt);

        \draw[thick] (0,0) circle (2);
    \end{scope}
    \begin{scope}[shift={(5,0)}] 
        \draw[olive,thick, postaction={decorate,decoration={markings,mark=at position 0.75 with {\arrow{>}}},decoration={markings,mark=at position 0.95 with {\node[rotate=50, text=olive, scale=0.8] at (0,0) {$\bowtie$};}}}] plot [smooth, tension=1] coordinates {(-1,0) (0,0.6) (1.2,0) (0,-0.6) (-1,0)};
        \node[olive,scale=1] at (0, -0.9) {$M_0$};
        
        \filldraw[black] (0,2) circle (2pt);
        \filldraw[black] (2,0) circle (2pt);
        \filldraw[black] (0,-2) circle (2pt);
        \filldraw[black] (-2,0) circle (2pt);

        \filldraw[black] (1,0) circle (2pt);
        \filldraw[black] (-1,0) circle (2pt);

        \draw[thick] (0,0) circle (2);
    \end{scope}
        \end{tikzpicture}
    \end{center}

    Our work in Example~\ref{running-exa} and the pictures above confirm that \[\dim_\k \Ext^1(M_0, P(3)) = \dim_\k \Ext^1(I(3),M_0) = 2 = \Int(P(3),M_0) = \Int(I(3),M_0) \] and that $\rho P(3) \in \TT, \rho^{-1}I(3) \in \TT,$ and $\rho^2 M_0 = M_0$.
\end{example}

\bibliography{references}
\bibliographystyle{amsplain}


\end{document}